\documentclass[11pt]{amsart}
\usepackage[top=1.5in, bottom=1.5in, left=1.5in, right=1.5in]{geometry}

\usepackage[all]{xy}
\usepackage[dvipsnames]{xcolor} 
\usepackage{graphicx} 
\usepackage{tikz}
\usetikzlibrary{patterns}
\usetikzlibrary{decorations.pathreplacing}
\usetikzlibrary{arrows}
\usetikzlibrary{decorations.markings}

\usepackage{subcaption}
\usepackage{amsmath}
\usepackage{amssymb}
\usepackage{enumerate}
\usepackage{mathdots}
\usepackage{mathtools}
\usepackage{diagbox}

\usepackage{hyperref}
\usepackage{amsthm}
\usepackage[capitalize,noabbrev]{cleveref}

\allowdisplaybreaks

\numberwithin{equation}{section}

\newtheorem{thm}{Theorem}[section]
\newtheorem{lemma}[thm]{Lemma}
\newtheorem{cor}[thm]{Corollary}
\newtheorem{prop}[thm]{Proposition}
\newtheorem{conj}[thm]{Conjecture}

\newtheorem{Definition}[thm]{Definition}

\newtheorem{Example}[thm]{Example}

\newtheorem{Remark}[thm]{Remark}
\newenvironment{remark}
  {\begin{Remark}\rm}{\hfill$\lozenge$\end{Remark}}

\usepackage{verbatim}

\newcommand{\abs}[1]{\left|#1\right|}

\DeclareMathOperator{\rev}{\mathbf{r}}
\DeclareMathOperator{\comp}{\mathbf{c}}
\DeclareMathOperator{\ASR}{ASR}
\DeclareMathOperator{\ASM}{ASM}

\newcommand{\Sch}[1]{\mathrm{Sch}_{#1}}
\newcommand{\s}[4]{S^{#1}_{{#2},{#3},{#4}}}
\newcommand{\asm}[2]{\ASM_{#1}^{#2}}
\newcommand{\asr}[2]{\ASR_{#1}^{#2}}

\newcommand{\emailhref}[1]{\email{\href{#1}{#1}}}
\title[Extended avoidance]{Pattern avoidance in alternating sign rectangles~I: Extended avoidance}

\author[H. H\"ongesberg]{Hans H\"ongesberg}
\address{Fakultät für Mathematik, Universität Wien, Vienna, Austria}
\emailhref{hans.hoengesberg@univie.ac.at}
\thanks{H.H\"ongesberg acknowledges support from the Austrian Science Fund (FWF) grant 10.55776/J4810 and from the University of Ljubljana, Faculty of Mathematics and Physics, for hosting him. He is also grateful to KTH Royal Institute of Technology for hosting him during a research visit and for providing accommodation and transportation.}

\author[M. Konvalinka]{Matja\v{z} Konvalinka}
\address{University of Ljubljana, Faculty of Mathematics and Physics and Institute for Mathematics, Physics and Mechanics, Ljubljana, Slovenia}
\thanks{M.Konvalinka was partially supported by ERC AdG KARST and by projects and programs J1-2452, N1-0218, P1-0294, P1-0297 of the Slovenian Research Agency.}

\author[S. Linusson]{Svante Linusson}
\address{Department of Mathematics, KTH Royal Institute of Technology, Stockholm, \mbox{Sweden}}
\thanks{S. Linusson was supported by the Swedish Science Council VR project 2022-03875}
\emailhref{linusson@kth.se}


\date{\today}

\begin{document}

\begin{abstract}
We introduce \emph{extendable pattern avoidance} for alternating sign
rectangles (ASRs), the natural rectangular generalization of
alternating sign matrices (ASMs). An ASR extendably avoids a pattern
$\pi$ if it is the upper left corner of an ASM avoiding $\pi$. 
For each of the four length-three patterns in the
equivalence class $\{312, 132, 213, 231\}$ we establish a complete
system of recurrence relations enumerating extendably
$\pi$-avoiding ASRs of size $r \times k$ with a prescribed number $d$
of nonempty rows. For $\pi = 312$ we further conjecture a closed-form
expression and prove it on several diagonal slices via bijections
involving Schr\"oder ballot numbers, refined Schr\"oder numbers and 
Delannoy paths that do not cross the main diagonal vertically. For ASRs of size $(r-1)\times (r+1)$  extendably avoiding $213$ we give a bijection to little Schröder paths of length~$r$. The remaining patterns of length three, $123$
and $321$, are more elusive, mirroring the situation of ASMs. 
\end{abstract}

\maketitle

\section{Introduction}

Alternating sign matrices (ASMs) were introduced by Mills, Robbins, and
Rumsey in the early 1980s in connection with Dodgson's condensation
algorithm for the evaluation of determinants. Their conjectured product
formula
\[
   %\lvert \mathrm{ASM}_n \rvert
   %\;=\; 
   \prod_{k=0}^{n-1} \frac{(3k+1)!}{(n+k)!}
\]
for the number of $n \times n$ ASMs was famously proved by Zeilberger~\cite{Zei96}; shortly afterwards, Kuperberg~\cite{Kup96}
gave another proof using the six-vertex model. Since then, ASMs have
become a central object in algebraic combinatorics, intertwined with
plane partitions  and a host of other
structures, see~\cite{Bre99} for an overview.

Pattern avoidance for permutations is a classical theme going back to
Knuth's analysis of stack-sortable permutations~\cite{Knu73}, with a very large literature by now.

Pattern avoidance was extended to ASMs by Johansson and the third author
in~\cite{JL07}. Identifying a pattern $\pi$ of length $n$ with its
permutation matrix, one says that an ASM $M$ of size $m \times m$
\emph{contains} $\pi$ if there are $1$s in $M$ arranged in the same
relative position as the $1$s of the permutation matrix of $\pi$, and
\emph{avoids} $\pi$ otherwise. The study of pattern avoidance in ASMs has then continued in other works, see for example \cite{BESST24}.

Meanwhile, the study of ASMs has been extended to non-square matrices. We define an \emph{alternating sign rectangle} (ASR) to be an
$r \times k$ matrix with entries in $\{-1, 0, 1\}$ satisfying the
natural sign-alternation and boundary conditions inherited from ASMs;
the precise definition is given below. Any ASR can be extended to a
square ASM by appending rows below and columns to the right. An ASR
\emph{extendably avoids} a pattern $\pi$ if at least one such extension
to a square ASM avoids $\pi$.
\Cref{subfig:ASR} illustrates the definition.

The goal of this paper is, for each length-three pattern
$\pi$, to study $S^{\pi}_{r,k,d}$, the number of ASRs of size $r \times k$ with $d$ nonempty rows
that extendably avoid $\pi$. We establish a complete recursive
characterization for every $\pi \in \{312, 132, 213, 231\}$, splitting
into the cases $d = r$ (row-complete), $d = k$ (balanced), and
$d < r, k$ (general). Together with the base cases provided by the
ASM enumeration of~\cite{JL07}, these recurrences determine
$S^{\pi}_{r,k,d}$ uniquely. For $\pi = 312$ we further
present a conjectural closed-form expression in \cref{conj:312-formula}, which we verify in the case of ASMs. We further give bijective proofs for some special cases
with classical lattice path families: Schr\"oder ballot numbers, refined Schr\"oder numbers
$a_{i,j}$, Delannoy paths that do not cross the main diagonal
vertically (sequence \cite[A268208]{oeis}), and little Schr\"oder
paths.

We also note that for the cross pattern $\chi$, the only $3 \times 3$ ASM that is
not a permutation matrix, every ASR containing a $-1$ is
automatically extendably $\chi$-containing, and the enumeration
reduces to counting $\{0,1\}$-matrices with at most one $1$ in each
row and column.

For the remaining patterns of length three, $123$ and $321$, we are
not able to provide a complete enumeration; this is unsurprising, as
even at the level of ASMs no closed form is known. Resolving these
cases for ASRs would in particular constitute a step towards
understanding the open problem of enumerating $123$-avoiding ASMs.

The paper is organized as follows. \Cref{sec:def} gives the necessary definitions. \Cref{sec:ASM} reviews the
necessary enumerative and bijective background on pattern avoidance in
ASMs. \Cref{sec:recursive} contains our main results: the complete
recursive characterizations of extendably $\pi$-avoiding ASRs for
$\pi \in \{312, 132, 213,\allowbreak 231\}$, including a conjectural closed form for
$\pi = 312$. \Cref{sec:bijections} provides enumerative and bijective consequences for some special cases. Finally, \cref{sec:conclusion} concludes with a brief
discussion of the patterns $123$ and $321$, as well as the cross pattern.

This is the first in a series of two papers. In~\cite{hkl}, we study \emph{classical} avoidance of alternating sign rectangles, especially the asymptotic behavior.

\section{Definitions}\label{sec:def}

Let $r$ and $k$ be positive integers. An \emph{alternating sign rectangle} (ASR) of size $r \times k$ is an $r \times k$ matrix with entries in $\{-1,0,1\}$ such that
\begin{itemize}
\item the nonzero entries alternate in each row and column;
\item if a row contains a nonzero entry, then its leftmost and rightmost nonzero entries are $1$;
\item if a column contains a nonzero entry, then its topmost nonzero entry is $1$.
\end{itemize}
A row or column containing no nonzero entries is called \emph{empty}. An ASR is \emph{row-complete} if it has no empty rows.  
If also all columns sum to $1$, an ASR is called \emph{balanced}.

Let us emphasize that rows and columns are \emph{not} equivalent in our definition of alternating sign rectangles; in particular, the transpose of an ASR is not necessarily an ASR.

ASRs are a generalization of (square) \emph{alternating sign matrices} (ASMs). An ASM is a square ASR whose rows and columns each sum to $1$. This implies that, in each row and in each column, there is at least one nonzero entry, and the first and the last nonzero entry in each row and each column has to be a $1$.

\begin{remark}
    There are a few generalizations of ASMs to rectangular shape in the literature. We want to compare these notions to ASRs. The term \emph{partial ASMs} was coined by Heuer and Striker \cite{HS22} and it essentially denotes topleft corners of ASMs. These matrices had previously been introduced simply as ASMs by Behrend \cite{Beh08} and independently up to a horizontal reflection as \emph{alternating matrices} by Fortin \cite{For08}. ASRs are a proper subset of partial ASMs since nonempty rows have to sum to $1$. In joint work with Billey \cite{BK26}, the second author introduced \emph{quilts of alternating sign matrices}. In that context, quilts of type~$(C_r,C_k)$ are called rectangular ASMs of size $r \times k$. If $r \le k$, these rectangular ASMs are in bijective correspondence with \emph{monotone triangles} with $r$ rows and positive entries less than or equal to $k$ as defined by Fischer \cite{Fis06}, and correspond to row-complete ASRs in our terminology.
\end{remark}

In this paper, we study \emph{extendable pattern avoidance} in ASRs. To this end, we think of ASRs to be extendable to a full square ASM downwards and to the right. See \cref{fig:extension} for an example.

\begin{figure}[htbp]
    \centering
    \begin{subfigure}[c]{.4\textwidth}
    \centering
        $\begin{pmatrix}
        0 & 1 & 0 & 0 & 0 & 0\\
        1 &-1 & 0 & 0 & 1 & 0\\
        0 & 0 & 0 & 1 &-1 & 1\\
        0 & 0 & 0 & 0 & 0 & 0\\
        0 & 0 & 0 & 0 & 1 & 0
    \end{pmatrix}$
    \caption{}
    \label{subfig:ASR}
    \end{subfigure}
    
    \begin{subfigure}[c]{.4\textwidth}
    \centering
        $\begin{pmatrix}
        0 & 1 & 0 & 0 & 0 & 0 & 0\\
        1 &-1 & 0 & 0 & 1 & 0 & 0\\
        0 & 0 & 0 & 1 &-1 & 1 & 0\\
        0 & 0 & 0 & 0 & 0 & 0 & 1\\
        0 & 0 & 0 & 0 & 1 & 0 & 0\\
        0 & 1 & 0 & 0 & 0 & 0 & 0\\
        0 & 0 & 1 & 0 & 0 & 0 & 0
    \end{pmatrix}$
    \caption{}
    \label{subfig:ASR_ext1}
    \end{subfigure}\qquad
    \begin{subfigure}[c]{.4\textwidth}
    \centering
        $\begin{pmatrix}
        0 & 1 & 0 & 0 & 0 & 0 & 0\\
        1 &-1 & 0 & 0 & 1 & 0 & 0\\
        0 & 0 & 0 & 1 &-1 & 1 & 0\\
        0 & 0 & 0 & 0 & 0 & 0 & 1\\
        0 & 0 & 0 & 0 & 1 & 0 & 0\\
        0 & 0 & 1 & 0 & 0 & 0 & 0\\
        0 & 1 & 0 & 0 & 0 & 0 & 0
        \end{pmatrix}$
    \caption{}
    \label{subfig:ASR_ext2}
    \end{subfigure}
    \caption{\cref{subfig:ASR} shows a $5 \times 6$ ASR with four nonempty rows. The extension in \cref{subfig:ASR_ext1} contains the pattern~$312$, whereas the extension in \cref{subfig:ASR_ext2} avoids it. Thus, the ASR in \cref{subfig:ASR} extendably avoids the pattern~$312$.}
    \label{fig:extension}
\end{figure}

Let $\pi$ be a permutation of length~$n$, and let $M$ be an ASM of size $m \times m$ such that $m \ge n$. We say that $M$ \emph{contains} a pattern $\pi$ if there exist $1 \le i_1 < \dots < i_n \le m$ and $1 \le j_1 < \dots < j_n \le m$ such that $M_{i_1,j_{\pi(1)}} = \dots = M_{i_n,j_{\pi(n)}} = 1$. Otherwise, $M$ \emph{avoids} the pattern~$\pi$.

It is convenient to interpret any pattern $\pi$ as its corresponding permutation matrix. Then $M$ avoids the pattern~$\pi$ if there are no $1$s in $M$ that are in the same relative position to each other as the $1$s in the permutation matrix of $\pi$.

An ASR is said to \emph{extendably avoid} a pattern if there exists an extension of the ASR to a square ASM that avoids the pattern. In other words, an ASR \emph{extendably contains} a pattern if every extension contains the pattern.
The notion of extended pattern avoidance was first introduced in \cite{EL95} by Eriksson and the third author and was further investigated in \cite{Lin02}.

\section{Pattern avoidance in alternating sign matrices}
\label{sec:ASM}

We first consider pattern avoidance in alternating sign matrices. To this end, let $\asm{n}{\pi}$ denote the set of $n\times n$ ASMs avoiding the pattern~$\pi$. If a pattern~$\tau$ is obtained from $\pi$ by rotation or reflection, then it holds by symmetry that $\abs{\asm{n}{\pi}} = \abs{\asm{n}{\tau}}$. Thus, it trivially follows that $\abs{\asm{n}{312}} = \abs{\asm{n}{132}} = \abs{\asm{n}{213}} = \abs{\asm{n}{231}}$ and $\abs{\asm{n}{123}} = \abs{\asm{n}{321}}$. It was shown in \cite{JL07} that $\asm{n}{312}$ is enumerated by the \emph{(large) Schr\"oder number} $\Sch{n-1}$. The $n$th Schr\"oder number $\Sch{n}$ counts the lattice paths from $(0,0)$ to $(n,n)$ with steps $(1,0)$, $(0,1)$ and $(1,1)$ that never go above the diagonal $y=x$. These paths are called \emph{(large) Schr\"oder paths}. More generally, we call any lattice path consisting of the steps $(1,0)$, $(0,1)$ and $(1,1)$ a \emph{Delannoy path}.

The first few Schr\"oder numbers are given by $\Sch{0}=1$, $\Sch{1}=2$, $\Sch{2}=6$, and $\Sch{3}=22$; see \cite[A006318]{oeis} for further values. Their generating function is $\sum_{n \ge 0} \Sch{n} x^n = \frac{1-x-\sqrt{1-6x+x^2}}{2x}$.
\cref{fig:Schroeder} shows the six lattice paths that are counted by $\Sch{2}$. In addition, we set $\Sch{-1}=1$, which will be useful later.

\begin{figure}[htbp]
    \centering
    \begin{tikzpicture}[scale = .5]
        \draw[thin, dotted] (0,0) grid (2,2);
        \draw[very thin, red] (0,0) -- (2,2);
        \draw[very thick] (0,0) -- (2,0) -- (2,2);
    \end{tikzpicture}\qquad\begin{tikzpicture}[scale = .5]
        \draw[thin, dotted] (0,0) grid (2,2);
        \draw[very thin, red] (0,0) -- (2,2);
        \draw[very thick] (0,0) -- (1,0) -- (1,1) -- (2,1) -- (2,2);
    \end{tikzpicture}\qquad\begin{tikzpicture}[scale = .5]
        \draw[thin, dotted] (0,0) grid (2,2);
        \draw[very thin, red] (0,0) -- (2,2);
        \draw[very thick] (0,0) -- (1,0) -- (1,1) -- (2,2);
    \end{tikzpicture}\qquad\begin{tikzpicture}[scale = .5]
        \draw[thin, dotted] (0,0) grid (2,2);
        \draw[very thin, red] (0,0) -- (2,2);
        \draw[very thick] (0,0) -- (1,0) -- (2,1) -- (2,2);
    \end{tikzpicture}\qquad\begin{tikzpicture}[scale = .5]
        \draw[thin, dotted] (0,0) grid (2,2);
        \draw[very thin, red] (0,0) -- (2,2);
        \draw[very thick] (0,0) -- (1,1) -- (2,1) -- (2,2);
    \end{tikzpicture}\qquad\begin{tikzpicture}[scale = .5]
        \draw[thin, dotted] (0,0) grid (2,2);
        \draw[very thin, red] (0,0) -- (2,2);
        \draw[very thick] (0,0) -- (2,2);
    \end{tikzpicture}
    \caption{The six Schr\"oder paths that are counted by $\Sch{2}$.}
    \label{fig:Schroeder}
\end{figure}
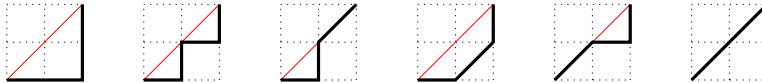

We recall the bijection~$\Psi$ between $312$-avoiding ASMs of size $n \times n$ and Schr\"oder paths from $(0,0)$ to $(n-1,n-1)$ presented in \cite{JL07}\footnote{In fact, the bijection in \cite{JL07} is defined for $132$-avoiding ASMs. Since the pattern $312$ is obtained from $132$ by reflecting across a vertical axis, it is straightforward to adapt the structural decomposition of $132$-avoiding ASMs to $312$-avoiding ASMs.}. We will make essential use of this construction later.

The bijection~$\Psi$ is defined recursively. For the base case $n=1$, the matrix~$(1)$ is mapped to the empty path. Now assume that $n>1$. There is exactly one $1$ in the first column of the $312$-avoiding ASM~$M$, say in row~$i$. We distinguish between two cases depending on whether row~$i$ contains a $-1$.

If row $n-1\ge i\ge 2$ does not, then $M$ decomposes into two smaller $312$-avoiding matrices $M_1 \in \asm{n-i}{312}$ and $M_2 \in \asm{i-1}{312}$ as shown in \cref{fig:ASM_bij_1}. The grey area represents $0$s. By recursion, we already know $\Psi(M_1)$ and $\Psi(M_2)$. To obtain $\Psi(M)$, we take $\Psi(M_1)$ with a diagonal step $(1,1)$ attached and then add $(1,0)$, $\Psi(M_2)$, and finally $(0,1)$.

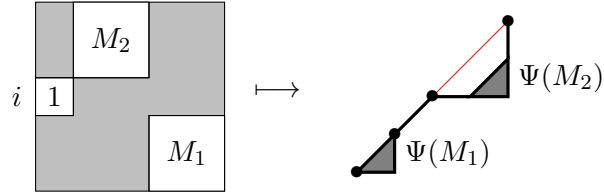
\begin{figure}[htbp]
    \centering
    \begin{tikzpicture}[scale=.5,baseline=(current bounding box.center)]
    \draw[fill=lightgray] (0,0) rectangle (5,5);

    \draw[fill=white] (0,2) rectangle (1,3);
    \draw[fill=white] (1,3) rectangle (3,5);
    \draw[fill=white] (3,0) rectangle (5,2);

    \draw (1,5) -- (1,3) -- (3,3) -- (3,5);
    \draw (3,0) -- (3,2) -- (5,2);
    \draw (0,2) -- (1,2) -- (1,3) -- (0,3);

    \node at (.5,2.5) {$1$};
    \node at (-.5,2.5) {$i$};
    \node at (4,1) {$M_1$};
    \node at (2,4) {$M_2$};
    \end{tikzpicture}$\quad\longmapsto\quad$
    \begin{tikzpicture}[scale=.5,baseline=(current bounding box.center)]      
    \draw[very thin, red] (0,0) -- (4,4);

    \draw[very thick] (1,1) -- (2,2);
    \draw[very thick] (2,2) -- (3,2);
    \draw[very thick] (4,4) -- (4,3);
    \draw[very thick, fill=gray] (0,0) -- (1,0) -- (1,1) -- cycle;
    \draw[very thick, fill=gray] (3,2) -- (4,2) -- (4,3) -- cycle;

    \node at (0,0) {$\bullet$};
    \node at (1,1) {$\bullet$};
    \node at (2,2) {$\bullet$};
    \node at (4,4) {$\bullet$};

    \node at (4,2.5)[right] {$\Psi(M_2)$};
    \node at (1,0.5)[right] {$\Psi(M_1)$};
 
    \end{tikzpicture}
    \caption{The bijection~$\Psi$ if row~$i$ contains no $-1$.}
    \label{fig:ASM_bij_1}
\end{figure}

If $i=1$, then $M_2=\emptyset$ and we obtain $\Psi(M)$ by adding just a diagonal step $(1,1)$ to $\Psi(M_1)$. If $i=n$, then $M_1=\emptyset$ and we do not have a diagonal step but obtain $\Psi(M)$ by  a step $(1,0)$ followed by $\Psi(M_2)$ and last an upstep $(0,1)$. 

If there is an additional $-1$ in row~$i\ge 2$, then $M$ decomposes as shown in \cref{fig:ASM_bij_2}, with $M_2 \in \asm{i-1}{312}$. The matrix $M_1$ is of size $(n-i+1) \times (n-i+1)$ and has a $-1$ in the top left corner. It becomes a $312$-avoiding ASM~$\widehat{M_1}$ if that $-1$ is changed into a $0$. To obtain $\Psi(M)$, we now take $\Psi(\widehat{M_1})$ (without a diagonal step attached) and add $(1,0)$, $\Psi(M_2)$, and $(0,1)$.

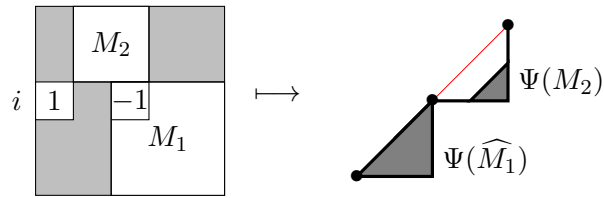
\begin{figure}[htbp]
    \centering
    \begin{tikzpicture}[scale=.5,baseline=(current bounding box.center)]
  \draw[fill=lightgray] (0,0) rectangle (5,5);

  \draw[fill=white] (0,2) rectangle (1,3);
  \draw[fill=white] (1,3) rectangle (3,5);
  \draw[fill=white] (2,0) rectangle (5,3);

  \draw (1,5) -- (1,3) -- (3,3) -- (3,5);
  \draw (2,0) -- (2,3) -- (5,3);
  \draw (2,2) -- (3,2) -- (3,3);

  \node at (.5,2.5) {$1$};
  \node at (-.5,2.5) {$i$};
  \node at (3.5,1.5) {$M_1$};
  \node at (2,4) {$M_2$};
  \node at (2.5,2.5) {$-1$};
    \end{tikzpicture}$\quad\longmapsto\quad$
    \begin{tikzpicture}[scale=.5,baseline=(current bounding box.center)]
    \draw[very thin, red] (0,0) -- (4,4);

    \draw[very thick] (1,1) -- (2,2);
    \draw[very thick] (2,2) -- (3,2);
    \draw[very thick] (4,4) -- (4,3);
    \draw[very thick, fill=gray] (0,0) -- (2,0) -- (2,2) -- cycle;
    \draw[very thick, fill=gray] (3,2) -- (4,2) -- (4,3) -- cycle;

    \node at (0,0) {$\bullet$};
    \node at (2,2) {$\bullet$};
    \node at (4,4) {$\bullet$};

    \node at (4,2.5)[right] {$\Psi(M_2)$};
    \node at (2,0.5)[right] {$\Psi(\widehat{M_1})$};
    \end{tikzpicture}
    \caption{The bijection~$\Psi$ if row~$i$ contains a $-1$.}
    \label{fig:ASM_bij_2}
\end{figure}

It is not difficult to see that this construction is reversible, for details see \cite{JL07}.

\section{Recursive enumeration of extended pattern avoidance in alternating sign rectangles}\label{sec:recursive}
In this section, we establish recurrence relations for the number of extendably $312$-, $132$-, $213$-, and $231$-avoiding ASRs with a given number of nonempty rows. To this end, we define $\asr{r,k,d}{\pi}$ to be the set of extendably $\pi$-avoiding ASRs of size $r\times k$ with $d$ nonempty rows, and we let $\s{\pi}{r}{k}{d}$ denote its cardinality. We readily see that $\asr{r,k,d}{\pi} = \varnothing$ if $r < d$ or $k < d$. Thus, we assume $r,k\ge d\ge 0$. In addition, we set $\s{\pi}{r}{k}{0} = 1$ for all $r$ and $k$.

To obtain recurrence relations for $\s{\pi}{r}{k}{d}$, we follow the same strategy for each pattern~$\pi$. First, we consider the case $d = r$, which means that there are no empty rows. We then consider the case $d = k$, implying that there are no empty columns. This follows from the fact that all rows are either empty or they sum to $1$. In addition, the topmost nonzero entry in each column is $1$. Consequently, there are at least as many nonempty columns as there are nonempty rows. If $d=k$ rows are nonempty, then at least $k$ columns are nonempty; therefore, there are no empty columns. Note, however, that %column-completeness 
no empty columns does not imply $d=k$ since columns do not necessarily sum to $1$. If we require the matrix to be balanced, then every column is nonempty and hence $d=k$.

Finally, we deal with $d < r, k$. Thus, we obtain three recurrence relations for every pattern. Together with the base cases $\s{\pi}{0}{k}{0} = 1 = \s{\pi}{r}{0}{0}$ and $\s{\pi}{n}{n}{n} = \Sch{n-1}$, these relations uniquely determine $\s{\pi}{r}{k}{d}$ for any $r,k\ge d\ge 0$ and any permutation $\pi \in \{312,132,213,231\}$.

Before turning to the individual cases, we state a simple but crucial observation.

\begin{lemma}
    Any ASR that extendably avoids a permutation of length~$3$ does not contain two $-1$s in the same row or column. 
\end{lemma}

\begin{proof}
    It is straightforward to see that any extension contains all of the patterns $312$, $132$, $213$, $231$, $123$, and $321$ as soon as there is one row or one column with at least two $-1$s.
\end{proof}

\subsection{\texorpdfstring{Extendably $\mathbf{312}$-avoiding ASRs}{Extendably 312-avoiding ASRs}}
In this section, we will focus on the pattern~$312$ which means avoiding $1$s relatively positioned as in the following matrix:
\[
\begin{array}{c@{\quad}l}
\begin{array}{c}
3 \\
1 \\
2 \\
\end{array}
&
\left(
\begin{array}{ccc}
0 & 0 & 1 \\
1 & 0 & 0 \\
0 & 1 & 0 \\
\end{array}
\right)
\end{array}
\]

First, we derive a recurrence relation for the case $d=r$, that is, when the ASR is row-complete.

\begin{prop}\label{lem:r=d}
    For $1\le r\le k$, we have 
    \begin{equation*}
    \s{312}{r}{k}{r}=\s{312}{r}{k-1}{r}+\s{312}{r-1}{k-1}{r-1}+\sum_{i=2}^{r} \Sch{i-2}\cdot \s{312}{r+1-i}{k+1-i}{r+1-i}.
    \end{equation*}
\end{prop}

\begin{proof}
    Consider a row-complete and extendably $312$-avoiding ASR of size $r \times k$. We distinguish cases according to the first column. If the first column is empty, we may always place a 1 in the extension in the first column below everything else without introducing an occurrence of the pattern~$312$. This yields the first term in the recursion.

    Now assume there is a $1$ in the first column, say in row~$i \ge 1$. There are two possibilities, either the remaining entries of row~$i$ are $0$s or there is a $-1$ in that row; note that the latter case can only occur if $i \neq 1$. See \cref{fig:312_d=r} for an illustration.

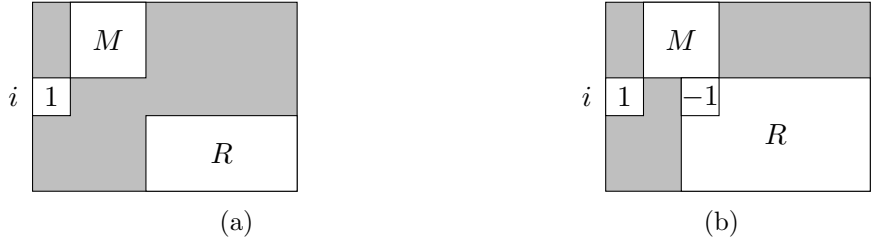
\begin{figure}[htbp]
\centering
\begin{subfigure}[c]{.45\textwidth}
\begin{tikzpicture}[scale=0.5]

    \draw[fill=lightgray] (0,0) rectangle (7,5);

    \draw[fill=white] (1,3) rectangle (3,5);
    \draw[fill=white] (0,2) rectangle (1,3);
    \draw[fill=white] (3,0) rectangle (7,2);

    \node at (.5,2.5) {$1$};
    \node at (-.5,2.5) {$i$};
    \node at (2,4) {$M$};
    \node at (5,1) {$R$};

\end{tikzpicture}
\caption{}
\label{subfig:312_d=r_a}
\end{subfigure}
\begin{subfigure}[c]{.45\textwidth}
\centering
\begin{tikzpicture}[scale=0.5]

    \draw[fill=lightgray] (0,0) rectangle (7,5);

    \draw[fill=white] (1,3) rectangle (3,5);
    \draw[fill=white] (0,2) rectangle (1,3);
    \draw[fill=white] (2,0) rectangle (7,3);
    \draw[fill=white] (2,2) rectangle (3,3);

    \node at (.5,2.5) {$1$};
    \node at (-.5,2.5) {$i$};
    \node at (2.5,2.5) {$-1$};
    \node at (2,4) {$M$};
    \node at (4.5,1.5) {$R$};

\end{tikzpicture}
\caption{}
\label{subfig:312_d=r_b}
\end{subfigure}
\caption{The structure of a row-complete and extendably $312$-avoiding ASR.}
\label{fig:312_d=r}
\end{figure}

First note that if there is a column with only $0$s in the top $i-1$ rows, then there must be a $1$ further down in that column, either in the ASR itself or in its extension. Thus, all other columns to its right must also be empty in the first $i-1$ rows.
That means that, beginning with the second column, all columns that are not empty in the first $i-1$ rows occur consecutively. If in one of these columns the bottommost nonzero entry in the top $i-1$ rows is a $-1$, then there will be a $1$ further down that column, again either in the ASR or in its extension. Since all rows sum to $1$, there is also a $1$ to the right of that $-1$. That would create a $312$-pattern with the $1$ in the $i$th row and the first column. Consequently, columns $2$ to $i$ will have a full $312$-avoiding ASM~$M$ in the first $i-1$ rows. We know there are $\Sch{i-2}$ such possibilities.

In the first case (\cref{subfig:312_d=r_a}), the bottom right rectangle~$R$ must form a row-complete and $312$-avoiding ASR of size $(r-i)\times (k-i)$, that is, $R\in\asr{r-i,k-i,r-i}{312}$. We therefore get the sum $\sum_{i=1}^k \Sch{i-2}\cdot \s{312}{r-i}{k-i}{r-i}$. In particular, if $i=1$, we obtain the term $\s{312}{r-1}{k-1}{r-1}$ since $\Sch{-1}=1$.

In the second case (\cref{subfig:312_d=r_b}), the rectangle~$R$ in the bottom right is one row and one column larger. It becomes an extendably $312$-avoiding ASR~$\widehat{R}$ if we replace the $-1$ in its top left corner by $0$.
Consequently, the possible $\widehat{R}$ are those in $\asr{r+1-i,k+1-i,r+1-i}{312}$ restricted to matrices without a $1$ in the top left corner, which are counted by $\s{312}{r-i}{k-i}{r-i}$. Thus, the second case contributes
\[
\sum_{i=2}^r \Sch{i-2}\cdot \left(\s{312}{r+1-i}{k+1-i}{r+1-i}-\s{312}{r-i}{k-i}{r-i}\right).
\]
Finally, note how all negative terms cancel with terms from the first case, and the proposition follows.

\end{proof}

Next, we study the case of extendably $312$-avoiding ASRs that are balanced, that is, $d=k$.

\begin{prop}\label{lem:k=d}
    For $1\le k<r$, we have
    \begin{equation}\label{eq:312_k=d}
    \s{312}{r}{k}{k}=\sum_{j=0}^k \s{312}{j}{k}{j}\binom{r-j-1}{k-j}.
    \end{equation}
\end{prop}

\begin{proof}
    Let $j+1$ be the first empty row of an extendably $312$-avoiding ASR, which is illustrated in \cref{fig:312_d=k}.
    
    \begin{figure}[htbp]
    \centering
    \begin{tikzpicture}[scale=0.5]

    \draw[fill=lightgray] (0,0) rectangle (5,7);
    \draw[fill=white]  (0,0) rectangle (5,3);
    \draw[fill=white]  (0,4) rectangle (5,7);

    \node at (-1,3.5) {$j+1$};
    \node at (2.5,5.5) {$R$};
    \node at (1,.5) {$1$};
    \node at (2.5,1.5) {$\iddots$};
    \node at (4,2.5) {$1$};

    \end{tikzpicture}
    \caption{The structure of a balanced %column-complete 
    and extendably $312$-avoiding ASR.}
    \label{fig:312_d=k}
    \end{figure}
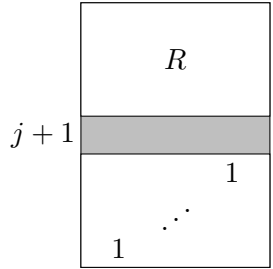
    
    In the first $j$ rows, there is a row-complete extendably $312$-avoiding ASR~$R$, hence $R \in \asr{j,k,j}{312}$, which yields $\s{312}{j}{k}{j}$ possibilities. When extending the ASR to an ASM, there will be a $1$ somewhere to the right in row~$j+1$. Thus, below that row in the ASR the $1$s must not form any $12$-pattern. For every choice of $k-j$ rows among the $r-1-j$ rows, there is exactly one way to fill them with $1$s to obtain an extendably $312$-avoiding ASR.
\end{proof}

We are now ready to treat the general case.

\begin{thm}\label{lem:general}
    For $d< r, k$ we have 
    \[
    \s{312}{r}{k}{d}=\s{312}{r}{k-1}{d}+\s{312}{r-1}{k-1}{d-1}+\sum_{i=2}^d \Sch{i-2}\cdot \left (\s{312}{r+1-i}{k+1-i}{d+1-i}-\binom{r-i}{d+1-i} \right ).
    \]
\end{thm}

\begin{proof}
    The proof of this theorem is very similar to the proof of \cref{lem:r=d}. The difference is that when there is a $1$ in row~$i$ of the first column and a $-1$ in the same row, the bottom right rectangle cannot have an empty first row. We must therefore subtract those ASRs counted by $\s{312}{r+1-i}{k+1-i}{d+1-i}$ whose first row is empty. By the same reasoning as in the proof of \cref{lem:k=d}, there is a unique way to fill in the $1$s for every choice of nonempty rows.
\end{proof}

Based on computer experiments, we conjecture the following closed-form formula for the number of extendably $312$-avoiding ASRs.

\begin{conj}\label{conj:312-formula}
For integers $d \leq r,k$, we have
\begin{multline}\label{eq:conj312}
    \s{312}{r}{k}{d} = \sum_{i=1}^{d} 2^{i-1} \binom{k}{i} \left( \binom{r-2}{d-i} - 2\binom{r-2}{d-i-2} \right) + \binom{r}{d}\\
    + \sum_{i \geq 0 }^{} C_{2i} \binom{r-1+2i}{d-1-2i} + 2 \sum_{i=1}^{d} (-1)^i \binom{r}{d-i},
\end{multline}
where $C_n = \frac{1}{n+1} \binom{2n}{n}$ denotes the $n$th Catalan number. 
\end{conj}

Let $F(r,k,d)$ denote the right-hand side of \eqref{eq:conj312}. We easily see that $F(0,k,0)=F(r,0,0)=1$ as well as $F(1,k,1)=k$ and $F(r,1,1)=r$. If $r=k=d=n \ge 1$, we obtain that $F(n,n,n)$ is equal to
\begin{equation*}
    \sum_{i=1}^{n} 2^{i-1} \binom{n}{i} \left( \binom{n-2}{n-i} - 2\binom{n-2}{n-i-2} \right)
    + \sum_{i \geq 0 }^{} C_{2i} \binom{n-1+2i}{n-1-2i} - 1
\end{equation*}
since
\begin{equation*}
 \sum_{j=1}^{n} (-1)^j \binom{n}{n-j}=\begin{cases} 0&\text{if }n=0,\\
 -1&\text{if }n \geq 1.
 \end{cases}
\end{equation*}
In the following, we show that $F(n,n,n) = \Sch{n-1}$. It was proved in \cite{Bar21} that the sum $\sum_{i \geq 0 }^{} C_{2i} \binom{n-1+2i}{n-1-2i}$ equals the $(n-1)$th \emph{little Schr\"oder number} (or \emph{Schr\"oder--Hipparchus number}) which is equal to half of the $(n-1)$th large Schr\"oder number.

For the remaining part, we write
\begin{multline*}
    \sum_{i=1}^{n} 2^{i-1} \binom{n}{i} \left( \binom{n-2}{n-i} - 2\binom{n-2}{n-i-2} \right)
    - 1\\
    = \sum_{i=1}^{n} 2^{i-1} \binom{n}{i} \binom{n-2}{n-i} -  \sum_{i=0}^{n} 2^{i} \binom{n}{i} \binom{n-2}{n-i-2}\\
    =n(n-1) {}_2F_1(-n+2,-n+2;3;2) - {}_2F_1(-n+2,-n;1;2)
\end{multline*}
by using hypergeometric notation\footnote{See \cite[Section~2]{AAR99} for a reference on hypergeometric functions and contiguous relations.}. By iterating Gau{\ss}' original contiguous relations, we obtain the identity
\begin{multline*}
    (1-a+c)(b+1)z {}_2F_1(a,b+2;c+2;z) - (c+z(1-a+b))(1+c) {}_2F_1(a,b+1;c+1;z)\\
    + c(1+c) {}_2F_1(a,b;c;z) = 0.
\end{multline*}
Setting $a=-n+2$, $b=-n$, $c=1$ and $z=2$, we have
\begin{multline*}
    n(n-1) {}_2F_1(-n+2,-n+2;3;2) - {}_2F_1(-n+2,-n;1;2) \\
    = {}_2F_1(-n+2,-n+1;2;2).
\end{multline*}
Since the $(n-1)$th little Schr\"oder number is given by  ${}_2F_1(-n+2,-n+1;2;2)$, it follows that $F(n,n,n)=\Sch{n-1}$.

\subsection{\texorpdfstring{Extendably $\mathbf{132}$-avoiding ASRs}{Extendably 132-avoiding ASRs}}

In this section, we will study the pattern $132$:
\[
\begin{array}{c@{\quad}l}
\begin{array}{c}
1 \\
3 \\
2 \\
\end{array}
&
\left(
\begin{array}{ccc}
1 & 0 & 0 \\
0 & 0 & 1 \\
0 & 1 & 0 \\
\end{array}
\right)
\end{array}
\]

We start %again with the case $1 \le r \le k$ and $d=r$, and 
with the following observation: Let $\pi$ be a permutation of length~$n$. The \emph{complement permutation $\pi^{\comp}$} is defined by $i \mapsto n+1-\pi(i)$. We obtain the permutation matrix of $\pi^{\comp}$ by reflecting the permutation matrix of $\pi$ across a vertical axis. Thus, we have $\s{132}{r}{k}{r} = \s{312}{r}{k}{r}$ if $1 \le r \le k$ since $312^{\comp}=132$. 

The remaining cases of  $1 \le k < r$ and $d=k$ as well as $d < r,k$ will be dealt with together in the following theorem. 

\begin{thm}
    For $d=k < r$ or $d< r, k$, we have 
    \[
    \s{132}{r}{k}{d}=\sum_{i=1}^{d+1} \Sch{i-2} \s{132}{r-i}{k+1-i}{d+1-i}.
    \]
\end{thm}

\begin{proof}
    Let $i$ be the first empty row of an extendably $132$-avoiding ASR. Since there must be a $1$ in this row of any extended ASM, the ASR looks as in \cref{fig:132_d=k}.
    
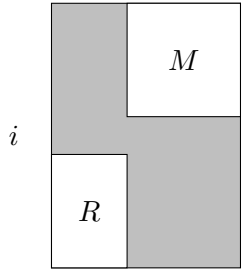
\begin{figure}[htbp]
\centering
\begin{tikzpicture}[scale=0.5]

\draw[fill=lightgray] (0,0) rectangle (5,7);
\draw[fill=white]  (0,0) rectangle (2,3);
\draw[fill=white]  (2,4) rectangle (5,7);

\node at (-1,3.5) {$i$};
\node at (3.5,5.5) {$M$};
\node at (1,1.5) {$R$};

\end{tikzpicture}
\caption{The structure of an extendably $132$-avoiding ASR with empty rows.}
\label{fig:132_d=k}
\end{figure}

    We see that there is $M \in \asm{i-1}{132}$ in columns $k+2-i$ to $k$ above row~$i$. This gives $\Sch{i-2}$ possibilities.

    Below row~$i$, columns $k+2-i$ to $k$ have to be empty. In the first $k+1-i$ columns, there can be any extendably $132$-avoiding ASR~$R$ with $d+1-i$ nonempty rows, for which there are $\s{132}{r-i}{k+1-i}{d+1-i}$ possibilities.
\end{proof}

\subsection{\texorpdfstring{Extendably $\mathbf{213}$-avoiding ASRs}{Extendably 213-avoiding ASRs}}
In this section, we will study the pattern $213$:
\[
\begin{array}{c@{\quad}l}
\begin{array}{c}
2 \\
1 \\
3 \\
\end{array}
&
\left(
\begin{array}{ccc}
0 & 1 & 0 \\
1 & 0 & 0 \\
0 & 0 & 1 \\
\end{array}
\right)
\end{array}
\]

Again, we start with row-complete ASRs, that is, $1 \le r \le k$ and $r=d$.

\begin{prop}\label{lem:213_r=d}
    For $1\le r\le k$ we have 
    \[
    \s{213}{r}{k}{r}=\s{213}{r}{k-1}{r}+\sum_{i=1}^r \Sch{i-1}\cdot \s{213}{r-i}{k-i}{r-i}.
    \]
\end{prop}
\begin{proof}
    Similarly to the case 312, we consider the first column. If it is empty, the remaining columns constitute an ASR in $\asr{r,k-1,r}{213}$.

    Now assume that there is a $1$ in row~$i$ of the first column. We distinguish two cases according to whether this $1$ is followed by a $-1$ in the same row, as illustrated in \cref{fig:213_d=r}.

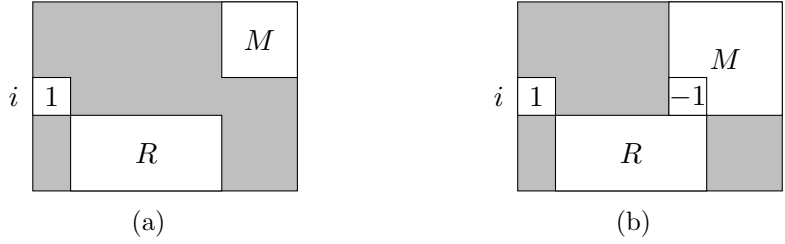
\begin{figure}[htbp]
\centering
\begin{subfigure}[c]{.45\textwidth}
\centering
\begin{tikzpicture}[scale=0.5]
    \draw[fill=lightgray] (0,0) rectangle (7,5);

    \draw[fill=white] (1,0) rectangle (5,2);
    \draw[fill=white] (5,3) rectangle (7,5);
    \draw[fill=white] (0,2) rectangle (1,3);

    \node at (.5,2.5) {$1$};
    \node at (-.5,2.5) {$i$};
    \node at (6,4) {$M$};
    \node at (3,1) {$R$};
\end{tikzpicture}
\caption{}
\label{subfig:213_d=r_a}
\end{subfigure}
\begin{subfigure}[c]{.45\textwidth}
\centering
\begin{tikzpicture}[scale=0.5]
    \draw[fill=lightgray] (0,0) rectangle (7,5);

    \draw[fill=white] (1,0) rectangle (5,2);
    \draw[fill=white] (4,2) rectangle (7,5);
    \draw[fill=white] (0,2) rectangle (1,3);
    \draw[fill=white] (4,2) rectangle (5,3);

    \node at (.5,2.5) {$1$};
    \node at (-.5,2.5) {$i$};
    \node at (4.5,2.5) {$-1$};
    \node at (5.5,3.5) {$M$};
    \node at (3,1) {$R$};
\end{tikzpicture}
\caption{}
\label{subfig:213_d=r_b}
\end{subfigure}
\caption{The structure of a row-complete and extendably $213$-avoiding ASR.}
\label{fig:213_d=r}
\end{figure}

If there is a column $j$ with only zeros in the top $i$ rows, then there must be a $1$ further down that column in any extension. Thus, all other columns before must also be empty in the first $i$ rows. As a consequence, if row $i$ comprises only one nonzero entry (\cref{subfig:213_d=r_a}), then there is a $213$-avoiding $M\in\asm{i-1}{213}$ in the last $i-1$ columns, enumerated by $\Sch{i-2}$. If row $i$ contains also a $-1$ (\cref{subfig:213_d=r_b}), then we take the $i \times i$ matrix~$M$ in the top right corner, replace the bottom left $-1$ by $0$ and obtain $\widehat{M}\in\asm{i}{213}$; note that $213$-avoiding ASMs of size $i \times i$ with a $0$ in the bottom left corner are enumerated by $\Sch{i-1} - \Sch{i-2}$.
In both cases, there is a $R \in \asr{r-i,k-i,r-i}{213}$ in columns $2$ to $k-i+1$ and the bottom $r-i$ rows.

Thus, we obtain
\[
    \s{213}{r}{k}{r}=\s{213}{r}{k-1}{r}+\sum_{i=1}^r \left( \Sch{i-2}\cdot \s{213}{r-i}{k-i}{r-i} + \left(\Sch{i-1} - \Sch{i-2}\right)\cdot \s{213}{r-i}{k-i}{r-i} \right),
\]
which concludes the proof.
\end{proof}

Let us now deal with balanced extendably $213$-avoiding ASRs.

\begin{prop} For $2 \leq k < r$, we have
    \begin{multline}\label{eq:balanced213}
    \s{213}{r}{k}{k}=\s{213}{r-1}{k-1}{k-1} + \sum_{i=1}^{r-k+1} \s{213}{r-i}{k-1}{k-1}\\
    + \sum_{i=r-k+2}^{r-1} \Sch{r-i-1} \left( \s{213}{i}{k-r+i}{k-r+i} - \binom{i-1}{k-r+i} \right),
    \end{multline}
    with the additional base case $\s{213}{r}{1}{1}=r$.
\end{prop}

\begin{proof}
    Let us assume $1 \leq k < r$ and consider a balanced extendably $213$-avoiding ASR of size $r \times k$ with $k$ nonempty rows. Since there are no empty columns, there is a unique $1$ in the first column, say in row~$i$. Now assume that all rows above are empty, see \cref{subfig:213_k=d_a}.
    
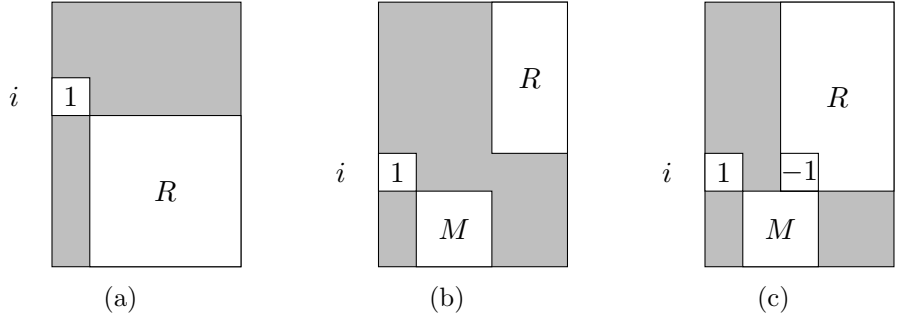
\begin{figure}[htbp]
\centering
\begin{subfigure}[c]{.3\textwidth}
\centering   
\begin{tikzpicture}[scale=0.5]

\draw[fill=lightgray] (0,0) rectangle (5,7);
\draw[fill=white] (0,4) rectangle (1,5);
\draw[fill=white] (1,0) rectangle (5,4);

\node at (-1,4.5) {$i$};
\node at (.5,4.5) {$1$};
\node at (3,2) {$R$};

\end{tikzpicture}
\caption{}
\label{subfig:213_k=d_a}
\end{subfigure}
\begin{subfigure}[c]{.3\textwidth}
\centering    
\begin{tikzpicture}[scale=0.5]

\draw[fill=lightgray] (0,0) rectangle (5,7);
\draw[fill=white] (0,2) rectangle (1,3);
\draw[fill=white] (1,0) rectangle (3,2);
\draw[fill=white] (3,3) rectangle (5,7);

\node at (-1,2.5) {$i$};
\node at (.5,2.5) {$1$};
\node at (4,5) {$R$};
\node at (2,1) {$M$};

\end{tikzpicture}
\caption{}
\label{subfig:213_k=d_b}
\end{subfigure}
\begin{subfigure}[c]{.3\textwidth}
\centering    
\begin{tikzpicture}[scale=0.5]

\draw[fill=lightgray] (0,0) rectangle (5,7);
\draw[fill=white] (0,2) rectangle (1,3);
\draw[fill=white] (1,0) rectangle (3,2);
\draw[fill=white] (2,2) rectangle (5,7);
\draw (2,2) rectangle (3,3);

\node at (-1,2.5) {$i$};
\node at (.5,2.5) {$1$};
\node at (2.5,2.5) {$-1$};
\node at (3.5,4.5) {$R$};
\node at (2,1) {$M$};

\end{tikzpicture}
\caption{}
\label{subfig:213_k=d_c}
\end{subfigure}
\caption{The structure of a balanced extendably $213$-avoiding ASR.}
\label{fig:213_k=d}
\end{figure}
    
    Thus, the rows below form an extendably $213$-avoiding ASR $R$ of size $(r-i) \times (k-1)$ with $k-1$ nonempty rows. Note that the $1$ in the first column can be at most in row~$r-k+1$ to make sure $R$ has at least $k-1$ rows. This gives 
    \begin{equation*}
    \sum_{i=1}^{r-k+1} \s{213}{r-i}{k-1}{k-1}
    \end{equation*}
    possibilities.

    If the rows above row~$i$ are not empty, then the extendably $213$-avoiding ASR has one of the forms in \cref{subfig:213_k=d_b} or \cref{subfig:213_k=d_c}.

If there is no $-1$ in row~$i$ (\cref{subfig:213_k=d_b}), then there is an $M\in \asm{r-i}{213}$ in rows $i+1$ to $r$ and columns $2$ to $r-i+1$. The reason is that an empty row below row~$i$ would create a $213$-pattern in the extended ASM. Consequently, $i$ ranges between $r-k+2$ and $r$. Above row~$i$, we have $R \in \asr{i-1,k-r+i-1,k-r+i-1}{213}$ in the top right corner. Thus, there are
\begin{equation*}
    \sum_{i=r-k+2}^{r} \Sch{r-i-1} \s{213}{i-1}{k-r+i-1}{k-r+i-1}
\end{equation*}
possibilities in total. Recall that we defined $\Sch{-1}=1$ for that purpose.

If there is a $-1$ in row~$i$ (\cref{subfig:213_k=d_c}), we again have a $213$-avoiding ASM~$M$ in rows $i+1$ to $r$ and columns $2$ to $r-i+1$. In the top right corner, we take the $i \times (k-r+i)$ matrix~$R$, replace the $-1$ at posititon~$(i,r-i+1)$ by $0$, and obtain $\widehat{R}\in\asr{i,k-r+i,k-r+i}{213}$. Since nonempty rows sum to $1$, we know there is somewhere a $1$ in row~$i$ in the last $k-r+i-1$ columns. The ASRs that have a $1$ in the bottom left corner are counted by $\s{213}{i-1}{k-r+i-1}{k-r+i-1}$ and those which have an empty bottom row by $\binom{i-1}{k-r+i}$ since they cannot have any $21$-pattern in the $i-1$ rows above. In total, we obtain
\begin{equation*}
    \sum_{i=r-k+2}^{r-1} \Sch{r-i-1} \left( \s{213}{i}{k-r+i}{k-r+i} - \s{213}{i-1}{k-r+i-1}{k-r+i-1} - \binom{i-1}{k-r+i} \right)
\end{equation*}
possibilities in that case. Note that here $i$ can be at most $r-1$ in order to ensure the right number of nonempty rows.

We finally add the three previous sums together. If $k \ge 2$, this sum simplifies to  \eqref{eq:balanced213}; if $k = 1$, we obtain $\s{213}{r}{1}{1}=r$.
\end{proof}

In the general case $d < r, k$, we argue similarly as before.

\begin{thm} For $d < r, k$, we have
    \begin{multline*}
    \s{213}{r}{k}{d}=
    \s{213}{r}{k-1}{d} + \sum_{i=1}^{r-d+1} \s{213}{r-i}{k-1}{d-1}\\
    + \sum_{i=r-d+2}^{r} \s{213}{r-i}{k-d+r-i}{r-i} \left( \s{213}{i}{d-r+i}{d-r+i} - \binom{i-1}{d-r+i} \right).
    \end{multline*}
\end{thm}

\begin{proof}
Consider an extendably $213$-avoiding ASR of size $r \times k$ with $d$ nonempty rows. We have $\s{213}{r}{k-1}{d}$ possibilities for the first column to be empty because we can always extend a $213$-avoiding ASM of size $n \times n$ by a $1$ in the bottom left corner to obtain a $213$-avoiding ASM of size $(n+1) \times (n+1)$.

Now assume that the first column is not empty and that there is a $1$ in row~$i$. If the rows above row~$i$ are empty, we have $\s{213}{r-i}{k-1}{d-1}$ possibilities to fill the rows below. To that end, the $1$ in the first column cannot be below row $r-d+1$.

If the rows above row~$i$ are not empty, we again distinguish between two cases, depending on whether row~$i$ contains a $-1$, see \cref{fig:213_gen}.

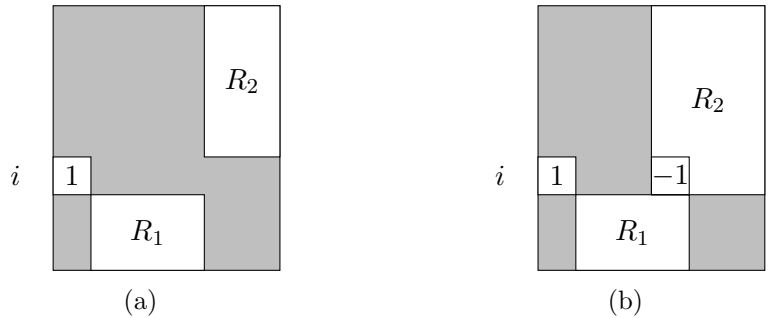
\begin{figure}[htbp]
\centering
    \begin{subfigure}[c]{.45\textwidth}
    \centering
    \begin{tikzpicture}[scale=0.5]
    \draw[fill=lightgray] (0,0) rectangle (6,7);
    \draw[fill=white] (0,2) rectangle (1,3);
    \draw[fill=white] (1,0) rectangle (4,2);
    \draw[fill=white] (4,3) rectangle (6,7);

    \node at (-1,2.5) {$i$};
    \node at (.5,2.5) {$1$};
    \node at (5,5) {$R_2$};
    \node at (2.5,1) {$R_1$};
    \end{tikzpicture}
    \caption{}
    \label{subfig:213_gen_a}
    \end{subfigure}
    \begin{subfigure}[c]{.45\textwidth}
    \centering
    \begin{tikzpicture}[scale=0.5]
    \draw[fill=lightgray] (0,0) rectangle (6,7);
    \draw[fill=white] (0,2) rectangle (1,3);
    \draw[fill=white] (1,0) rectangle (4,2);
    \draw[fill=white] (3,2) rectangle (6,7);
    \draw (3,2) rectangle (4,3);

    \node at (-1,2.5) {$i$};
    \node at (.5,2.5) {$1$};
    \node at (3.5,2.5) {$-1$};
    \node at (4.5,4.5) {$R_2$};
    \node at (2.5,1) {$R_1$};
    \end{tikzpicture}
    \caption{}
    \label{subfig:213_gen_b}
    \end{subfigure}
\caption{The structure of an extendably $213$-avoiding ASR.}
\label{fig:213_gen}
\end{figure}

In both cases, $R_1$ is a row-complete, extendably $213$-avoiding ASR and $R_2$ is a balanced extendably $213$-avoiding ASR (after possibly replacing the $-1$ in the bottom left corner by a $0$). This implies that $i$ has to be at least $r-d+2$.

If row~$i$ does not contain a $-1$ (\cref{subfig:213_gen_a}), then $R_1$ is of size $(r-i) \times (k-d+r-i)$ 
and $R_2$ is of size $(i-1) \times (d-r+i-1)$, which gives
\begin{equation*}
    \sum_{i=r-d+2}^{r} \s{213}{r-i}{k-d+r-i}{r-i} \cdot \s{213}{i-1}{d-r+i-1}{d-r+i-1}
\end{equation*}
possibilities.

On the other hand, if row~$i$ contains a $-1$ (\cref{subfig:213_gen_b}), $R_1$ is again of size $(r-i) \times (k-d+r-i)$. Furthermore, $R_2$ is of size $i \times (d-r+i)$. Since we have to replace the $-1$ by a $0$ in order to obtain a proper extendably $213$-avoiding ASR, we have to count all ASRs of size $i \times (d-r+i)$ that do not have a $1$ in the bottom left corner and that do not have an empty bottom row. In total, we get
\begin{equation*}
    \sum_{i=r-d+2}^{r} \s{213}{r-i}{k-d+r-i}{r-i} \left( \s{213}{i}{d-r+i}{d-r+i} - \s{213}{i-1}{d-r+i-1}{d-r+i-1} - \binom{i-1}{d-r+i} \right)
\end{equation*}
possibilities in that case.

By summing all possibilities, we again obtain the desired identity. 
\end{proof}

\subsection{\texorpdfstring{Extendably $\mathbf{231}$-avoiding ASRs}{Extendably 231-avoiding ASRs}}

In this section, we will study the pattern~$231$:
\[
\begin{array}{c@{\quad}l}
\begin{array}{c}
2 \\
3 \\
1 \\
\end{array}
&
\left(
\begin{array}{ccc}
0 & 1 & 0 \\
0 & 0 & 1 \\
1 & 0 & 0 \\
\end{array}
\right)
\end{array}
\]

We first observe that, by reflecting across a vertical axis, it follows that $\s{231}{r}{k}{r} = \s{213}{r}{k}{r}$ if $1 \le r \le k$. The other two cases we need to consider, namely $d=k < r$ and $d < r, k$, are very similar to the case of the pattern~$132$.

\begin{thm} For $d=k < r$ or $d < r, k$, we have
    \[
    \s{231}{r}{k}{d}=\sum_{i=1}^{d+1} \Sch{i-2} \s{231}{r-i}{k-i+1}{d-i+1}.
    \]
\end{thm}

\begin{proof}
    Assume that row~$i$ is the first empty row of an extendably $231$-avoiding ASR of size $r \times k$. Since there will be a $1$ somewhere in row~$i$ of any extension, the ASR looks as in \cref{fig:231_d=k}.

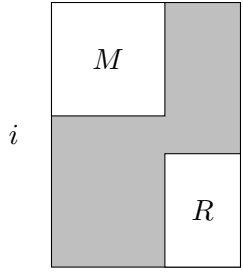
\begin{figure}[htbp]
\centering
\begin{tikzpicture}[scale=0.5]

\draw[fill=lightgray] (0,0) rectangle (5,7);
\draw[fill=white]  (3,0) rectangle (5,3);
\draw[fill=white]  (0,4) rectangle (3,7);

\node at (-1,3.5) {$i$};
\node at (1.5,5.5) {$M$};
\node at (4,1.5) {$R$};

\end{tikzpicture}
\caption{The structure of an extendably $231$-avoiding ASR with empty rows.}
\label{fig:231_d=k}
\end{figure}

The top left corner $M$ is a $231$-avoiding ASM of size $(i-1) \times (i-1)$ and $R$ is an extendably $231$-avoiding ASR of size $(r-i) \times (k-i+1)$ with $d-i+1$ nonempty rows. The assertion then follows immediately.
\end{proof}

\section{Bijections and combinatorial interpretations}\label{sec:bijections}

In this section, we present closed-form expressions for some special cases and provide bijections with classical families of lattice paths, including refined Schr\"oder numbers as well as Schr\"oder ballot numbers, Delannoy paths that do not cross the main diagonal vertically, and little Schr\"oder paths. 

\subsection{\texorpdfstring{$312$-avoiding ASRs and refined and generalized Schröder numbers}{312-avoiding ASRs and refined and generalized Schröder numbers}}

For $j\ge 2, i\ge 1$, let $a_{i,j}$ denote the number of Schröder paths from $(0,0)$ to $(i,i)$ that touch the diagonal exactly $j$ times. Note that any Schröder path touches the diagonal at the starting and ending point, therefore $j\ge 2$ and $\Sch{i}=\sum_{j=2}^{i+1} a_{i,j}$. For example, $a_{i,i+1}=2^i$ for $i \ge 1$, $a_{i,2}=\Sch{i-1}$ for $i \ge 2$, and $a_{i,j} = 0$ if $j > i+1$. We also have the following recursion by summing over the first point of return to the diagonal $(m,m)$; note the special case if $m=1$:
\begin{equation} \label{Sch_refinement}
    a_{i,j}=2a_{i-1,j-1} +\sum_{m=2}^{i}\Sch{m-1}\cdot a_{i-m,j-1}
    =a_{i-1,j-1} +\sum_{m=1}^{i}\Sch{m-1}\cdot a_{i-m,j-1}.
\end{equation}

The refined Schr\"oder numbers $a_{i,j}$ have been studied before, see \cite[A108891]{oeis}. They explicitily appear in \cite[Proposition 6.6]{BMO25}, although they seem to have already appeared earlier (in disguise), for example in \cite{PS98}.

We can link the refined Schr\"oder numbers $a_{i,j}$ to row-complete, extendably $312$-avoiding ASRs through the following theorem.

\begin{thm}\label{thm:r=d}
    For $2 \le r \le k$, we have
    \begin{equation}\label{eq:thm_r=d}
    \s{312}{r}{k}{r}=\sum_{j=2}^{r} a_{r-1,j}\binom{k-r+j}{j}.
    \end{equation}
\end{thm}
\begin{proof}
    This identity is proven by induction on $r$. For the base case, \cref{lem:r=d} implies $\s{312}{1}{k}{1}=k$ and consequently $\s{312}{2}{k}{2}=(k-1)k$, which coincides with
    \[
    \sum_{j=2}^2 a_{1,j}\binom{k-2+j}{j}= a_{1,2}\binom{k}{2} = (k-1)k.
    \]
    For the induction step, we assume $r \ge 3$ and deduce from \cref{lem:r=d} that
    \begin{multline*}
        \s{312}{r}{k}{r} - \s{312}{r}{k-1}{r} = \s{312}{r-1}{k-1}{r-1} + \sum_{i=2}^{r} \Sch{i-2} \cdot \s{312}{r+1-i}{k+1-i}{r+1-i}\\
        = \sum_{j=2}^{r-1} a_{r-2,j} \binom{k-r+j}{j} + \sum_{i=2}^{r-1} \Sch{i-2} \sum_{j=2}^{r+1-i} a_{r-i,j}\binom{k-r+j}{j} + \Sch{r-2} (k-r+1)\\
        = \sum_{j=2}^{r-1} \binom{k-r+j}{j} \left( a_{r-2,j} + \sum_{i=2}^{r+1-j} \Sch{i-2} a_{r-i,j} \right) + \Sch{r-2} (k-r+1),
    \end{multline*}
    where we keep in mind that \eqref{eq:thm_r=d} only holds true for $r \ge 2$. By using identity~\eqref{Sch_refinement} and
    \begin{equation*}
        \sum_{i=2}^{r+1-j} \Sch{i-2} a_{r-i,j} = \sum_{i=2}^{r} \Sch{i-2} a_{r-i,j}
    \end{equation*}
    because $a_{r-i,j} = 0$ if $j > r-i+1$, we then obtain
    \begin{equation*}
        \sum_{j=2}^{r-1} \binom{k-r+j}{j} a_{r-1,j+1} + \Sch{r-2} (k-r+1).
    \end{equation*}

    On the other hand, \eqref{eq:thm_r=d} is equivalent to
    \begin{multline*}
        \s{312}{r}{k}{r} - \s{312}{r}{k-1}{r} = \sum_{j=2}^{r} a_{r-1,j}\binom{k-r+j-1}{j-1}\\
        = \sum_{j=3}^{r} a_{r-1,j}\binom{k-r+j-1}{j-1} +  a_{r-1,2}\binom{k-r+1}{1}\\
        = \sum_{j=2}^{r-1} a_{r-1,j+1}\binom{k-r+j}{j} +  \Sch{r-2} (k-r+1),
    \end{multline*}
    where we used $a_{i,2}=\Sch{i-1}$ for $i \ge 2$. This concludes the proof.
\end{proof}

We now consider the following generalization of Schr\"oder paths: lattice paths from $(0,0)$ to $(n,m)$ with $n \ge m$, using steps $(1,0)$, $(0,1)$, and $(1,1)$, that never go above the diagonal $y=x$. Let $\Sch{n,m}$ denote the number of all these paths. We call these numbers \emph{Schr\"oder ballot numbers} in reference to the classical ballot numbers that generalize Catalan numbers. It follows from \cite[Theorem~2.9]{Sch25} (by setting $l=1$) that $\Sch{n,m}$ is given by
\begin{equation}\label{eq:Schroeder_ballot_numbers}
\frac{n-m+1}{n+1} \sum_{v=0}^{m} \binom{n+1}{m-v} \binom{n+v}{v}.
\end{equation}
We can express these Schr\"oder ballot numbers in terms of numbers of row-complete and extendably $312$-avoiding ASRs.

\begin{thm}\label{thm:Schroeder_ballot}
For $1 \le r \le k$, we have
\begin{equation*}
\Sch{k-1,r-1} = \s{312}{r}{k}{r}-\s{312}{r}{k-1}{r}.
\end{equation*}
\end{thm}

\begin{proof}
It is easy to see that $\Sch{n,0}=1$, $\Sch{n,n}=\Sch{n}$ and $\Sch{n,m}=\Sch{n,m-1}+\Sch{n-1,m}+\Sch{n-1,m-1}$ for $n \ge m$ since $\Sch{n-1,m}=0$ if $m=n$.

We show that the numbers $f(r,k) \coloneqq \s{312}{r}{k}{r}-\s{312}{r}{k-1}{r}$ fulfil the same recursion. It follows from \cref{thm:r=d} that $f(r,k) %= \s{312}{r}{k}{r}-\s{312}{r}{k-1}{r}
= \sum_{j=2}^r a_{r-1,j}\binom{k-r+j-1}{j-1}$ if $r \ge 2$.

For the initial conditions, we can easily see that $f(1,k)=1$ and $f(r,r)=\sum_{j=2}^{r} a_{r-1,j} = \Sch{r-1}$. For the recursion relation, we consider the set of all Schr\"oder paths from $(0,0)$ to $(r-1,r-1)$, which we uniquely partition into segments such that each segment touches the diagonal exactly twice. If $j$ is the number of times the Schr\"oder path touches the diagonal, then we have $j-1$ segments. For a given $k \ge r$, we then label each segment by an integer from $1$ to $k-r+j-1$ such that the labels strictly increase from the start point to the end point of the Schr\"oder path. These labelled paths are enumerated by $f(r,k)$. With this combinatorial interpretation in mind, we prove that $f(r,k)=f(r,k-1)+f(r-1,k)+f(r-1,k-1)$. We partition our labelled Schr\"oder paths as follows: If none of the segments has the highest possible label $k-r+j-1$, then these paths are enumerated by $f(r,k-1)$. Note that this set of paths is empty if $k=r$, matching $f(r,r-1)=0$. Now assume that the last segment is decorated by the highest possible label $k-r+j-1$. If this segment consists of a simple diagonal step $(1,1)$, we remove this step to obtain a Schr\"oder path from $(0,0)$ to $(r-2,r-2)$ whose segments are increasingly labelled with integers from $\{1,\dots,k-r+j-2\}$; these paths are enumerated by $f(r-1,k-1)$. If the last segment of the path is not $(1,1)$, then it starts with a horizontal unit step and ends with a vertical unit step. We remove these two steps from the last segment and, provided that the remaining part is not empty, add the rest back again. Thus, we add new segments to the path, say $t$, to which we assign the consecutive labels $k-r+j,\dots,k-r+j+t-1$. The labelled Schr\"oder path has now $j+t-2$ segments. With the maximal label being $k-r+j+t-1 = k-(r-1)+(j+t-2)$ we obtain a class of labelled Schr\"oder paths enumerated by $f(r-1,k)$.
This proves $f(r,k)=f(r,k-1)+f(r-1,k)+f(r-1,k-1)$.
\end{proof}

Using \cref{thm:Schroeder_ballot} and \eqref{eq:Schroeder_ballot_numbers}, we obtain the following double sum for $\s{312}{r}{k}{r}$.

\begin{cor}\label{cor:312_d=r}
For $1 \le r \le k$, we have
\begin{equation*}
\s{312}{r}{k}{r} = \sum_{i=r-1}^{k-1} \frac{i-r+2}{i+1} \sum_{v=0}^{r-1} \binom{i+1}{r-1-v} \binom{i+v}{v}.
\end{equation*}
\end{cor}

Applying \cref{cor:312_d=r} to \cref{lem:k=d} yields the following expression for the number of balanced extendably $312$-avoiding ASRs.

\begin{cor}
    For $0\le k<r$, we have
    \begin{equation*}
    \s{312}{r}{k}{k}=\binom{r-1}{k}+\sum_{j=1}^k \binom{r-j-1}{k-j} \sum_{i=j-1}^{k-1} \frac{i-j+2}{i+1} \sum_{v=0}^{j-1} \binom{i+1}{j-1-v} \binom{i+v}{v}.
    \end{equation*}
\end{cor}

By recasting \eqref{eq:312_k=d}, we obtain another presentation for the  number of balanced extendably $312$-avoiding ASRs, namely
\[
\s{312}{r}{k}{k}=\sum_{j=0}^k c_{k,j}\binom{r-j}{k-j},
\]
where $c_{k,j}=\s{312}{j}{k}{j}-\s{312}{j-1}{k}{j-1}$. Then $c_{k,j}$ yields a refinement of $\Sch{k-1}$ since
$\sum_{j=0}^{k-1} c_{k,j} = \Sch{k-1}$ as shown in \cref{tab:ckj}. We leave it as an open problem to determine a natural combinatorial interpretation for this refinement.

\begin{table}[htbp]
$\begin{array}{r|rrrrr||r}
\hbox{\diagbox{$k$}{$j$}} & 0 & 1 & 2 & 3 & 4 & \Sigma\\
\hline
1 & 1 &   &    &    &    & 1 \\
2 & 1 & 1 &    &    &    & 2  \\
3 & 1 & 2 & 3  &    &    & 6  \\
4 & 1 & 3 & 8  & 10 &    & 22 \\
5 & 1 & 4 & 15 & 32 & 38 & 90 
\end{array}$
\caption{The first few values of $c_{k,j}$.}
\label{tab:ckj}
\end{table}

\subsection{\texorpdfstring{$312$-avoiding ASRs and restricted Delannoy paths}{312-avoiding ASRs and restricted Delannoy paths}}

We provide the first few values of $\s{312}{r}{k}{r}$ in \cref{tab:312}.

\begin{table}[htbp]
$\begin{array}{r|rrrrrrrrrrr}
\hbox{\diagbox{$r$}{$k$}} & 0 & 1 & 2 & 3 & 4 & 5 & 6 & 7 & 8 & 9 & 10 \\
\hline
0 & 1 & 1 & 1 & 1 & 1 & 1 & 1 & 1 & 1 & 1 & 1 \\
1 &   & 1 & 2 & 3 & 4 & 5 & 6 & 7 & 8 & 9 & 10 \\
2 &   &   & 2 & 6 & 12 & 20 & 30 & 42 & 56 & 72 & 90 \\
3 &   &   &   & 6 & 22 & 52 & 100 & 170 & 266 & 392 & 552 \\
4 &   &   &   &   & 22 & 90 & 236 & 500 & 930 & 1582 & 2520 \\
5 &   &   &   &   &    & 90 & 394 & 1108 & 2516 & 5006 & 9086 \\
\end{array}$ 
\caption{The first few values of $\s{312}{r}{k}{r}$.}
\label{tab:312}
\end{table}

The diagonal $\left( \s{312}{r}{r+2}{r} \right)_{r \ge 0}=1,3,12,52,236,1108,\dots$ from \cref{tab:312} appears in \cite[A268208]{oeis} as the number of Delannoy paths from $(0,0)$ to $(r,r)$ that do not cross the diagonal $y=x$ vertically (see also \cite{PR19}). A Delannoy path is said to cross the main diagonal vertically if it enters the diagonal by a vertical step and, after possibly some diagonal steps, leaves it by a vertical step; see \cref{fig:royal_paths} for an example. 

\begin{figure}[htbp]
    \centering
    \begin{subfigure}[c]{0.75\textwidth}
    \centering
    \begin{tikzpicture}[scale = .5]
        \draw[thin, dotted] (0,0) grid (2,2);
        \draw[very thin, red] (0,0) -- (2,2);
        \draw[very thick] (0,0) -- (0,2) -- (2,2);
    \end{tikzpicture}\qquad\begin{tikzpicture}[scale = .5]
        \draw[thin, dotted] (0,0) grid (2,2);
        \draw[very thin, red] (0,0) -- (2,2);
        \draw[very thick] (0,0) -- (0,1) -- (1,2) -- (2,2);
    \end{tikzpicture}\qquad\begin{tikzpicture}[scale = .5]
        \draw[thin, dotted] (0,0) grid (2,2);
        \draw[very thin, red] (0,0) -- (2,2);
        \draw[very thick] (0,0) -- (0,1) -- (1,1) -- (1,2) -- (2,2);
    \end{tikzpicture}\qquad\begin{tikzpicture}[scale = .5]
        \draw[thin, dotted] (0,0) grid (2,2);
        \draw[very thin, red] (0,0) -- (2,2);
        \draw[very thick] (0,0) -- (0,1) -- (1,1) -- (2,2);
    \end{tikzpicture}\qquad\begin{tikzpicture}[scale = .5]
        \draw[thin, dotted] (0,0) grid (2,2);
        \draw[very thin, red] (0,0) -- (2,2);
        \draw[very thick] (0,0) -- (0,1) -- (2,1) -- (2,2);
    \end{tikzpicture}\qquad\begin{tikzpicture}[scale = .5]
        \draw[thin, dotted] (0,0) grid (2,2);
        \draw[very thin, red] (0,0) -- (2,2);
        \draw[very thick] (0,0) -- (1,1) -- (1,2) -- (2,2);
    \end{tikzpicture}

    \bigskip

    \begin{tikzpicture}[scale = .5]
        \draw[thin, dotted] (0,0) grid (2,2);
        \draw[very thin, red] (0,0) -- (2,2);
        \draw[very thick] (0,0) -- (2,2);
    \end{tikzpicture}\qquad\begin{tikzpicture}[scale = .5]
        \draw[thin, dotted] (0,0) grid (2,2);
        \draw[very thin, red] (0,0) -- (2,2);
        \draw[very thick] (0,0) -- (1,1) -- (2,1) -- (2,2);
    \end{tikzpicture}\qquad\begin{tikzpicture}[scale = .5]
        \draw[thin, dotted] (0,0) grid (2,2);
        \draw[very thin, red] (0,0) -- (2,2);
        \draw[very thick] (0,0) -- (1,0) -- (1,1) -- (2,2);
    \end{tikzpicture}\qquad\begin{tikzpicture}[scale = .5]
        \draw[thin, dotted] (0,0) grid (2,2);
        \draw[very thin, red] (0,0) -- (2,2);
        \draw[very thick] (0,0) -- (1,0) -- (1,1) -- (2,1) -- (2,2);
    \end{tikzpicture}\qquad\begin{tikzpicture}[scale = .5]
        \draw[thin, dotted] (0,0) grid (2,2);
        \draw[very thin, red] (0,0) -- (2,2);
        \draw[very thick] (0,0) -- (1,0) -- (2,1) -- (2,2);
    \end{tikzpicture}\qquad\begin{tikzpicture}[scale = .5]
        \draw[thin, dotted] (0,0) grid (2,2);
        \draw[very thin, red] (0,0) -- (2,2);
        \draw[very thick] (0,0) -- (2,0) -- (2,2);
    \end{tikzpicture}
    \caption{}
    \label{subfig:royal_paths}
    \end{subfigure}\qquad\begin{subfigure}[c]{.125\textwidth}
    \centering
    \begin{tikzpicture}[scale = .5]
        \draw[thin, dotted] (0,0) grid (2,2);
        \draw[very thin, red] (0,0) -- (2,2);
        \draw[very thick] (0,0) -- (1,0) -- (1,2) -- (2,2);
    \end{tikzpicture}
    \caption{}
    \label{subfig:royal_paths_except}
    \end{subfigure}
    \caption{\cref{subfig:royal_paths} shows the twelve Delannoy paths from the origin to $(2,2)$ that do not cross the main diagonal vertically. \cref{subfig:royal_paths_except} shows the only Delannoy path from the origin to $(2,2)$ that crosses the main diagonal vertically, as it enters and exits the diagonal via vertical steps.}
    \label{fig:royal_paths}
\end{figure}
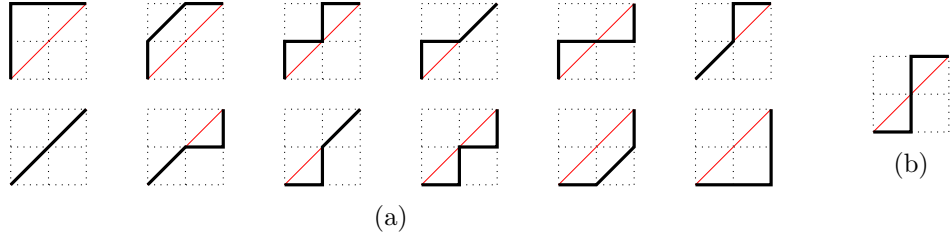

This is not a coincidence as we provide a bijection that maps row-complete, extendably $312$-avoiding ASRs of size $r \times (r+2)$ to Delannoy paths from $(0,0)$ to $(r,r)$ that do not cross the diagonal $y=x$ vertically. The bijection is based on the map~$\Psi$ presented in \cref{sec:ASM}.

\begin{thm}
    Let $r \ge 0$ be an integer. Then $\s{312}{r}{r+2}{r}$ equals the number of Delannoy paths from $(0,0)$ to $(r,r)$ that do not cross the diagonal $y=x$ vertically.
\end{thm}

\begin{proof}

The bijection~$\Psi$ presented in \cref{sec:ASM} maps $n \times n$ $312$-avoiding ASMs to Schr\"oder paths from $(0,0)$ to $(n-1,n-1)$. Let $\widetilde{\Psi}$ be the corresponding bijection that maps a $312$-avoiding  ASMs $M$ to the path obtained by reflecting $\Psi(M)$ along the diagonal~$y=x$; we call these paths \emph{superdiagonal Schr\"oder paths} in contrast to the \emph{subdiagonal Schr\"oder paths} defined earlier.

Let us now recursively construct a bijection~$\Xi$ between $\asr{r,r+2,r}{312}$ and Delannoy paths from $(0,0)$ to $(r,r)$ that do not cross the diagonal~$y=x$ vertically. To this end, consider a matrix $R \in \asr{r,r+2,r}{312}$.

We consider two base cases. If $r=0$, we map the empty ASR to the empty path. The base case~$r=1$ comprises the following three mappings:
    \begin{center}
    \begin{tikzpicture}[scale=0.5]
    \draw node at (1,3) {$\begin{pmatrix}
    1 & 0 & 0
    \end{pmatrix}$};
    \draw node at (1,2) {\rotatebox{-90}{$\mapsto$}};
    \draw node at (1,1) {\begin{tikzpicture}[scale=.5]
    \draw[thin, dotted] (0,0) grid (1,1);
    \draw[very thin, red] (0,0) -- (1,1);
    \draw[very thick] (0,0) -- (1,1);
    \end{tikzpicture}};

    \draw node at (7,3) {$\begin{pmatrix}
    0 & 1 & 0
    \end{pmatrix}$};
    \draw node at (7,2) {\rotatebox{-90}{$\mapsto$}};
    \draw node at (7,1) {\begin{tikzpicture}[scale=.5]
    \draw[thin, dotted] (0,0) grid (1,1);
    \draw[very thin, red] (0,0) -- (1,1);
    \draw[very thick] (0,0) -- (0,1) -- (1,1);
    \end{tikzpicture}};

    \draw node at (13,3) {$\begin{pmatrix}
    0 & 0 & 1
    \end{pmatrix}$};
    \draw node at (13,2) {\rotatebox{-90}{$\mapsto$}};
    \draw node at (13,1) {\begin{tikzpicture}[scale=.5]
    \draw[thin, dotted] (0,0) grid (1,1);
    \draw[very thin, red] (0,0) -- (1,1);
    \draw[very thick] (0,0) -- (1,0) -- (1,1);
    \end{tikzpicture}};
    
    \end{tikzpicture}
    \end{center}

Now assume $r>1$. We distinguish the following three cases:
    \begin{enumerate}[(i)] 
        \item {\bf{The first column is nonempty and has a $1$ in row~$i$ (\cref{fig:312-ASR_bij_1}).}} We can decompose $R$ into $M \in \asm{i-1}{312}$ and another submatrix~$R^\prime$. If row~$i$ does not contain a $-1$ (\cref{subfig:312-ASR_bij_1a}), then $R^\prime \in \asr{r-i,r+2-i,r-i}{312}$. In this case, we first map $R$ to $\Xi(R^{\prime})$, which starts at $(0,0)$ and ends at $(r-i,r-i)$; we then append a diagonal step followed by a horizontal step. Next, we join the lattice path of $\Psi(M)$, which starts at $(r-i+2,r-i+1)$ and ends at $(r,r-1)$, and finally add a vertical step. In particular, if $i=1$, then $\Xi(R)$ ends with a diagonal step.

    \begin{figure}[htbp]
    \centering
    \begin{subfigure}[c]{.75\textwidth}
    \centering
    \begin{tikzpicture}[scale=.5,baseline=(current bounding box.center)]
    \draw[fill=lightgray] (0,0) rectangle (7,5);

    \draw[fill=white] (1,3) rectangle (3,5);
    \draw[fill=white] (3,0) rectangle (7,2);
    \draw[fill=white] (0,2) rectangle (1,3);

    \node at (.5,2.5) {$1$};
    \node at (-.5,2.5) {$i$};
    \node at (2,4) {$M$};
    \node at (5,1) {$R^{\prime}$};
    \end{tikzpicture}$\quad\longmapsto\quad$
    \begin{tikzpicture}[scale=.5,baseline=(current bounding box.center)]      
    \draw[very thin, red] (0,0) -- (5,5);

    \draw[very thick, fill=gray] (0,0) rectangle (2,2);
    \draw[very thick] (2,2) -- (3,3) -- (4,3);
    \draw[very thick,fill=gray] (4,3) -- (5,3) -- (5,4) -- cycle;
    \draw[very thick] (5,4) -- (5,5);
    \node at (0,0) {$\bullet$};
    \node at (2,2) {$\bullet$};
    \node at (3,3) {$\bullet$};
    \node at (5,5) {$\bullet$};
    \node at (2,1)[right] {$\Xi(R^{\prime})$};
    \node at (5,3.5)[right] {$\Psi(M)$};
    \end{tikzpicture}
    \caption{}
    \label{subfig:312-ASR_bij_1a}
    \end{subfigure}

    \begin{subfigure}[c]{.75\textwidth}
    \centering
    \begin{tikzpicture}[scale=.5,baseline=(current bounding box.center)]
    \draw[fill=lightgray] (0,0) rectangle (7,5);

    \draw[fill=white] (1,3) rectangle (3,5);
    \draw[fill=white] (2,0) rectangle (7,3);
    \draw[fill=white] (0,2) rectangle (1,3);
    \draw[fill=white] (2,2) rectangle (3,3);

    \node at (.5,2.5) {$1$};
    \node at (2.5,2.5) {$-1$};
    \node at (-.5,2.5) {$i$};
    \node at (2,4) {$M$};
    \node at (4.5,1.5) {$R^{\prime}$};
    \end{tikzpicture}$\quad\longmapsto\quad$
    \begin{tikzpicture}[scale=.5,baseline=(current bounding box.center)]      
    \draw[very thin, red] (0,0) -- (5,5);

    \draw[very thick, fill=gray] (0,0) rectangle (3,3);
    \draw[very thick] (3,3) -- (4,3);
    \draw[very thick,fill=gray] (4,3) -- (5,3) -- (5,4) -- cycle;
    \draw[very thick] (5,4) -- (5,5);
    \node at (0,0) {$\bullet$};
    \node at (3,3) {$\bullet$};
    \node at (5,5) {$\bullet$};
    \node at (3,1.5)[right] {$\Xi(\widehat{R^{\prime}})$};
    \node at (5,3.5)[right] {$\Psi(M)$};
    \end{tikzpicture}
    \caption{}
    \label{subfig:312-ASR_bij_1b}
    \end{subfigure}
    \caption{The bijection~$\Xi$ if the first column is nonempty.}
    \label{fig:312-ASR_bij_1}
    \end{figure}
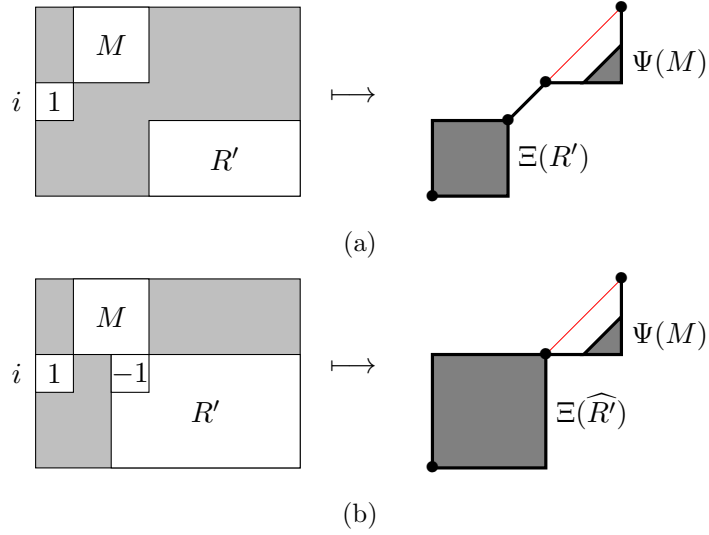

        If row~$i$ contains a $-1$ (\cref{subfig:312-ASR_bij_1b}), then $R^\prime$ is of size $(r+1-i) \times (r+3-i)$. We obtain a matrix $\widehat{R^\prime} \in \asr{r+1-i,r+3-i,r+1-i}{312}$ if we replace that $-1$ in $R^\prime$ by $0$. In this case, we first map $R$ to $\Xi(\widehat{R^{\prime}})$, which starts at $(0,0)$ and ends at $(r-i+1,r-i+1)$; we then append a horizontal step. Next, we join the lattice path of $\Psi(M)$, which starts at $(r-i+2,r-i+1)$ and ends at $(r,r-1)$; finally, we add a vertical step.
        \item {\bf{The first two columns are empty (\cref{fig:312-ASR_bij_2}).}} After removing the two empty columns, we obtain a matrix $M \in \asm{r}{312}$. We map $R$ to the lattice path that starts with a horizontal step, continues with $\Psi(M)$ from $(1,0)$ to $(r,r-1)$, and ends with a vertical step.

    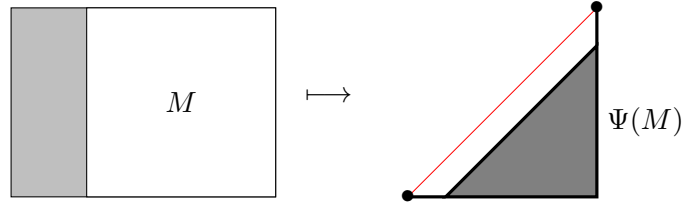
\begin{figure}[htbp]
    \centering
    \begin{tikzpicture}[scale=.5,baseline=(current bounding box.center)]
    \draw[fill=lightgray] (0,0) rectangle (7,5);

    \draw[fill=white] (2,0) rectangle (7,5);

    \node at (4.5,2.5) {$M$};
    \end{tikzpicture}$\quad\longmapsto\quad$
    \begin{tikzpicture}[scale=.5,baseline=(current bounding box.center)]      
    \draw[very thin, red] (0,0) -- (5,5);

    \draw[very thick] (0,0) -- (1,0);
    \draw[very thick, fill=gray] (1,0) -- (5,0) -- (5,4) -- cycle;
    \draw[very thick] (5,4) -- (5,5);
    \node at (0,0) {$\bullet$};
    \node at (5,5) {$\bullet$}; 
    \node at (5,2)[right] {$\Psi(M)$};
    \end{tikzpicture}
    \caption{The bijection~$\Xi$ if the first two columns are empty.}
    \label{fig:312-ASR_bij_2}
    \end{figure}
        \item {\bf{The first column is empty, but the second is not (\cref{fig:312-ASR_bij_3}).}} These matrices are mapped to superdiagonal Schr\"oder paths that end with a horizontal step. These are counted by \cite[A006319]{oeis} and first appeared in \cite{Kre73}. We distinguish three subcases.

    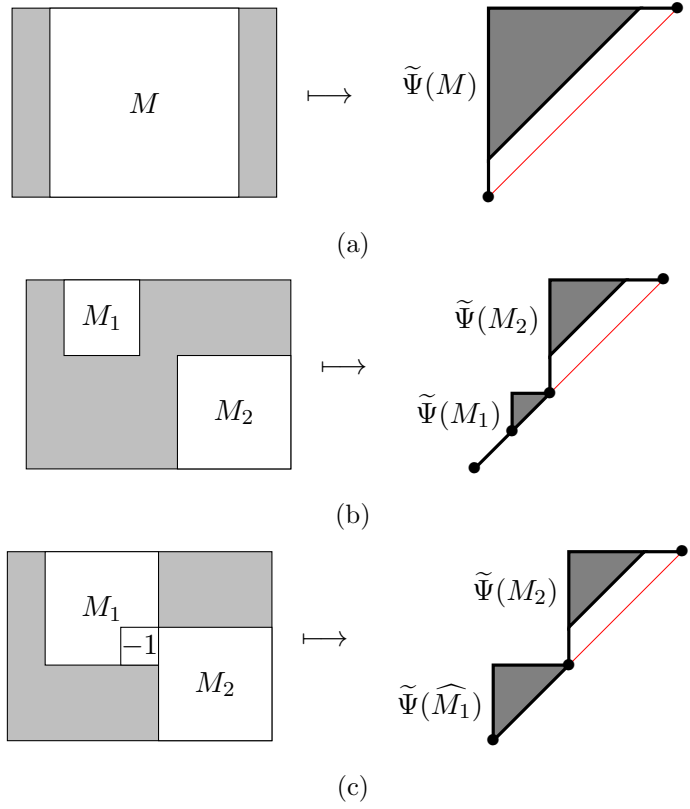
\begin{figure}[htbp]
    \centering
    \begin{subfigure}[c]{.75\textwidth}
    \centering
    \begin{tikzpicture}[scale=.5,baseline=(current bounding box.center)]
        \draw[fill=lightgray] (0,0) rectangle (7,5);

        \draw[fill=white] (1,0) rectangle (6,5);

        \node at (3.5,2.5) {$M$};
    \end{tikzpicture}$\quad\longmapsto\quad$
    \begin{tikzpicture}[scale=.5,baseline=(current bounding box.center)]      
        \draw[very thin, red] (0,0) -- (5,5);

        \draw[very thick] (0,0) -- (0,1);
        \draw[very thick, fill=gray] (0,1) -- (0,5) -- (4,5) -- cycle;
        \draw[very thick] (4,5) -- (5,5);

        \node at (0,0) {$\bullet$};
        \node at (5,5) {$\bullet$};
    
        \node at (0,3)[left] {$\widetilde{\Psi}(M)$};
    \end{tikzpicture}
    \caption{}
    \label{subfig:312-ASR_bij_3a}
    \end{subfigure}
    
    \begin{subfigure}[c]{.75\textwidth}
    \centering
    \begin{tikzpicture}[scale=.5,baseline=(current bounding box.center)]
        \draw[fill=lightgray] (0,0) rectangle (7,5);

        \draw[fill=white] (1,3) rectangle (3,5);
        \draw[fill=white] (4,0) rectangle (7,3);

        \node at (2,4) {$M_1$};
        \node at (5.5,1.5) {$M_2$};

        \end{tikzpicture}$\quad\longmapsto\quad$
    \begin{tikzpicture}[scale=.5,baseline=(current bounding box.center)]      
        \draw[very thin, red] (0,0) -- (5,5);

        \draw[very thick] (0,0) -- (1,1);
        \draw[very thick, fill=gray] (1,1) -- (2,2) -- (1,2) -- cycle;
        \draw[very thick] (2,2) -- (2,3);
        \draw[very thick, fill=gray] (2,3) -- (4,5) -- (2,5) -- cycle;
        \draw[very thick] (4,5) -- (5,5);

        \node at (0,0) {$\bullet$};
        \node at (1,1) {$\bullet$};
        \node at (2,2) {$\bullet$};
        \node at (5,5) {$\bullet$};
    
        \node at (2,4)[left] {$\widetilde{\Psi}(M_2)$};
        \node at (1,1.5)[left] {$\widetilde{\Psi}(M_1)$};
    \end{tikzpicture}
    \caption{}
    \label{subfig:312-ASR_bij_3b}
    \end{subfigure}

    \begin{subfigure}[c]{.75\textwidth}
    \centering
    \begin{tikzpicture}[scale=.5,baseline=(current bounding box.center)]
        \draw[fill=lightgray] (0,0) rectangle (7,5);

        \draw[fill=white] (1,2) rectangle (4,5);
        \draw[fill=white] (3,2) rectangle (4,3);
        \draw[fill=white] (4,0) rectangle (7,3);

        \node at (2.5,3.5) {$M_1$};
        \node at (5.5,1.5) {$M_2$};

        \node at (3.5,2.5) {$-1$};
        \end{tikzpicture}$\quad\longmapsto\quad$
    \begin{tikzpicture}[scale=.5,baseline=(current bounding box.center)]      
    \draw[very thin, red] (0,0) -- (5,5);

    \draw[very thick, fill=gray] (0,0) -- (2,2) -- (0,2) -- cycle;
    \draw[very thick] (2,2) -- (2,3);
    \draw[very thick, fill=gray] (2,3) -- (4,5) -- (2,5) -- cycle;
    \draw[very thick] (4,5) -- (5,5);

    \node at (0,0) {$\bullet$};
    \node at (2,2) {$\bullet$};
    \node at (5,5) {$\bullet$};
    
    \node at (2,4)[left] {$\widetilde{\Psi}(M_2)$};
    \node at (0,1)[left] {$\widetilde{\Psi}(\widehat{M_1})$};
    \end{tikzpicture}
    \caption{}
    \label{subfig:312-ASR_bij_3c}
    \end{subfigure}
    \caption{The bijection~$\Xi$ if the first column is empty but not the second.}
    \label{fig:312-ASR_bij_3}
    \end{figure}

        If $R$ has two empty columns, namely the first and the last column (\cref{subfig:312-ASR_bij_3a}), we remove these empty columns to obtain an ASM~$M\in\asm{r}{312}$. We map $R$ to the lattice path that starts with a vertical step, followed by $\widetilde{\Psi}(M)$ from $(0,1)$ to $(r-1,r)$, and ends with a horizontal step.

        If $R$ has two empty columns, whereas the second empty column is at position~$j \neq r+2$ (\cref{subfig:312-ASR_bij_3b}) (these matrices are enumerated by \cite[A006319]{oeis}), then $j$ encodes the point of the last return in the corresponding lattice path~$\Xi(R)$. More precisely, $R$ decomposes into two $312$-avoiding ASMs $M_1 \in \asm{j-2}{312}$ and $M_2 \in \asm{r-j+2}{312}$. We map $R$ to the lattice path that starts with a diagonal step, followed by $\widetilde{\Psi}(M_1)$ from $(1,1)$ to $(j-2,j-2)$; we then append a vertical step, $\widetilde{\Psi}(M_2)$ from $(j-2,j-1)$ to $(r-1,r)$, and finally a horizontal step.

    If $R$ has only one empty column (\cref{subfig:312-ASR_bij_3c}) (these matrices are counted by \cite[A006320]{oeis}), then there is exactly one column that sums to $0$ besides the first column, say at position~$j$. Then $R$ decomposes into matrices $M_1$ of size $(j-1) \times (j-1)$ with a $-1$ in the bottom right corner and $M_2 \in \asm{r-j+2}{312}$. When changing the $-1$ in the bottom right corner of $M_1$ to $0$, we obtain $\widehat{M_1} \in \asm{j-1}{312}$. In particular, $\widehat{M_1}$ has no $1$ in the bottom right corner.

    Here $j$ encodes again the point of the last return to the main diagonal in $\Xi(R)$. We map $R$ to the lattice path starting with  $\widetilde{\Psi}(\widehat{M_1})$ from $(0,0)$ to $(j-2,j-2)$ (the fact that $\widehat{M_1}$ has no $1$ in the bottom right corner ensures that $\widetilde{\Psi}(\widehat{M_1})$ starts with a vertical step), which is followed by a vertical step, $\widetilde{\Psi}(M_2)$ from $(j-2,j-1)$ to $(r-1,r)$, and a horizontal step.

\end{enumerate}
\end{proof}

As a consequence of \cite{PR19}, we obtain the following closed-form expression.

\begin{cor}
    For any nonnegative integer~$r$, we have 
    \begin{equation*}
         \s{312}{r}{r+2}{r} = \sum_{i=0}^{r} \frac{1}{i+1} \binom{2i+2}{i} \binom{r+i}{r-i}.
    \end{equation*}
\end{cor}

\subsection{\texorpdfstring{$213$-avoiding ASRs and little Schröder paths}{213-avoiding ASRs and little Schröder paths}}

We provide the first few values of row-complete, extendably $213$-avoiding ASRs in \cref{tab:213}.

\begin{table}[htbp]
$\begin{array}{r|rrrrrrrrrrr}
\hbox{\diagbox{$r$}{$k$}} & 0 & 1 & 2 & 3 & 4 & 5 & 6 & 7 & 8 & 9 & 10 \\
\hline
0 & 1 & 1 & 1 & 1 & 1 & 1 & 1 & 1 & 1 & 1 & 1 \\
1 &   & 1 & 2 & 3 & 4 & 5 & 6 & 7 & 8 & 9 & 10 \\
2 &   &   & 2 & 6 & 11 & 17 & 24 & 32 & 41 & 51 & 62 \\
3 &   &   &   & 6 & 22 & 45 & 76 & 116 & 166 & 227 & 300 \\
4 &   &   &   &   & 22 & 90 & 197 & 353 & 569 & 857 & 1230 \\
5 &   &   &   &   &    & 90 & 394 & 903 & 1688 & 2833 & 4436 \\
\end{array}$
\caption{The first few values of $\s{213}{r}{k}{r}$.}
\label{tab:213}
\end{table}

We observe that the diagonal $\left(\s{213}{r}{r+2}{r}\right)_{r \ge 0}$ in \cref{tab:213} comprises the little Schröder numbers. These numbers count Schr\"oder paths that do not have diagonal steps on the main diagonal. We call these paths \emph{little Schr\"oder paths} in contrast to the ordinary (large) Schr\"oder paths.

\begin{thm}
Let $r \ge 1$ be an integer. Then $\s{213}{r-1}{r+1}{r-1}$ equals the $r$th little Schr\"oder number.
\end{thm}

\begin{proof}
    We prove this statement by providing a bijection~$\Phi$ from extendably $213$-avoiding ASRs in $\asr{r-1,r+1,r-1}{213}$ to little Schr\"oder paths of length~$r$. To this end, let us first consider how to apply the map $\Psi$ to $213$-avoiding ASMs. For a permutation~$\pi$ of length~$n$, let $\pi^{\rev}$ denote the \emph{reverse permutation}, defined by $i \mapsto \pi(n+1-i)$. The reverse permutation corresponds to a reflection of the corresponding permutation matrix across a horizontal axis. This map naturally generalizes to all other matrices. Since $213^{\rev}=312$, we readily obtain a bijection between $213$-avoiding ASMs $M$ and large Schr\"oder paths defined by $M \mapsto \Psi(M^{\rev})$.

    We define the bijection~$\Phi$ recursively. For the base case~$r=1$, we map the empty matrix to the unique little Schr\"oder path of length~$1$:

    \begin{center}
    \begin{tikzpicture}[scale=0.5]
    \draw node at (0,0) {$\varnothing$};
    \draw node at (1,0) {$\mapsto$};
    \draw node at (2,0) {\begin{tikzpicture}[baseline=(current bounding box.center),scale=.5]
    \draw[thin, dotted] (0,0) grid (1,1);
    \draw[very thin, red] (0,0) -- (1,1);
    \draw[very thick] (0,0) -- (1,0) -- (1,1);
    \end{tikzpicture}};
    \end{tikzpicture}
    \end{center}

    Now consider a matrix $R\in\asr{r-1,r+1,r-1}{213}$ with $r>1$. If the first column is empty, then the matrix with its first column removed has exactly one column that sums to zero and can thus be uniquely completed to a $213$-avoiding ASM~$M$ of size $r \times r$ by adding a new bottom row. In this case, we map $R$ to the lattice path that starts at $(0,0)$ with a horizontal step to the right, which is followed by $\Psi(M^{\rev})$ from $(1,0)$ to $(r,r-1)$ and by a vertical step from $(r,r-1)$ to $(r,r)$.

    If the first column is not empty, then it must contain a $1$, say in row~$i$. We have one of the two cases shown in \cref{fig:213-ASR_bij}.

    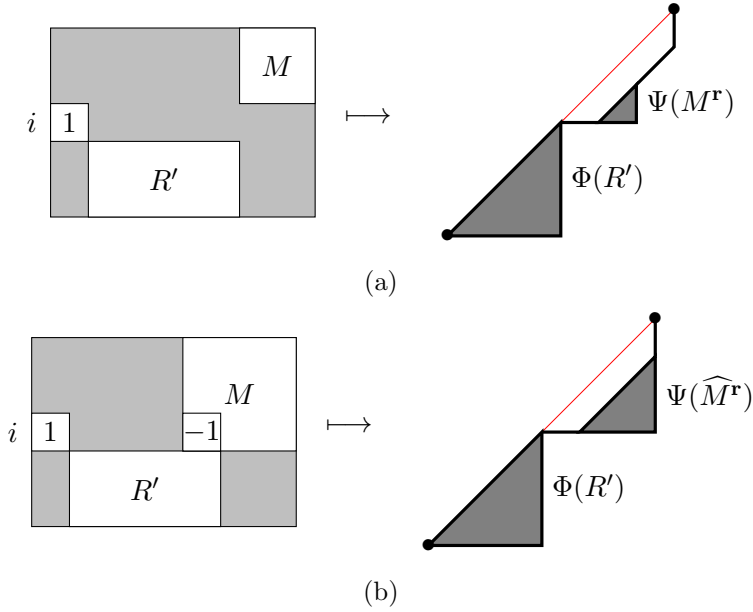
\begin{figure}[htbp]
    \centering
    \begin{subfigure}[c]{.75\textwidth}
    \centering
    \begin{tikzpicture}[scale=0.5,baseline=(current bounding box.center)]
    \draw[fill=lightgray] (0,0) rectangle (7,5);

    \draw[fill=white] (1,0) rectangle (5,2);
    \draw[fill=white] (5,3) rectangle (7,5);
    \draw[fill=white] (0,2) rectangle (1,3);

    \node at (.5,2.5) {$1$};
    \node at (-.5,2.5) {$i$};
    \node at (6,4) {$M$};
    \node at (3,1) {$R^{\prime}$};
    \end{tikzpicture}$\quad\longmapsto\quad$
    \begin{tikzpicture}[scale=.5,baseline=(current bounding box.center)]      
    \draw[very thin, red] (0,0) -- (6,6);

    \draw[very thick,fill=gray] (0,0) -- (3,0) -- (3,3) -- cycle;
    \draw[very thick] (3,3) -- (4,3);
    \draw[very thick,fill=gray] (4,3) -- (5,3) -- (5,4) -- cycle;
    \draw[very thick] (5,4) -- (6,5) -- (6,6);
    \node at (0,0) {$\bullet$};
    \node at (6,6) {$\bullet$};
    \node at (3,1.5)[right] {$\Phi(R^{\prime})$};
    \node at (5,3.5)[right] {$\Psi(M^{\rev})$};
    \end{tikzpicture}
    \caption{}
    \label{subfig:213-ASR_bij_a}
    \end{subfigure}

    \begin{subfigure}[c]{.75\textwidth}
    \centering
    \begin{tikzpicture}[scale=0.5,baseline=(current bounding box.center)]
    \draw[fill=lightgray] (0,0) rectangle (7,5);

    \draw[fill=white] (1,0) rectangle (5,2);
    \draw[fill=white] (4,2) rectangle (7,5);
    \draw[fill=white] (0,2) rectangle (1,3);
    \draw[fill=white] (4,2) rectangle (5,3);

    \node at (.5,2.5) {$1$};
    \node at (-.5,2.5) {$i$};
    \node at (4.5,2.5) {$-1$};
    \node at (5.5,3.5) {$M$};
    \node at (3,1) {$R^{\prime}$};
    \end{tikzpicture}$\quad\longmapsto\quad$
    \begin{tikzpicture}[scale=.5,baseline=(current bounding box.center)]      
    \draw[very thin, red] (0,0) -- (6,6);

    \draw[very thick,fill=gray] (0,0) -- (3,0) -- (3,3) -- cycle;
    \draw[very thick] (3,3) -- (4,3);
    \draw[very thick,fill=gray] (4,3) -- (6,3) -- (6,5) -- cycle;
    \draw[very thick] (6,5) -- (6,6);
    \node at (0,0) {$\bullet$};
    \node at (6,6) {$\bullet$};
    \node at (3,1.5)[right] {$\Phi(R^{\prime})$};
    \node at (6,4)[right] {$\Psi(\widehat{M}^{\rev})$};
    \end{tikzpicture}
    \caption{}
    \label{subfig:213-ASR_bij_b}
    \end{subfigure}
    \caption{The bijection~$\Phi$ if the first column is nonempty.}
    \label{fig:213-ASR_bij}
    \end{figure}

    The matrix decomposes into $M \in \asm{i-1}{213}$ and $R^{\prime} \in \asr{r-i-1,r-i+1,r-i-1}{213}$ if row~$i$ contains no other nonzero entry (\cref{subfig:213-ASR_bij_a}). The bijection maps $R$ to $\Phi(R^{\prime})$, which starts at $(0,0)$ and ends at $(r-i,r-i)$; it is then successively followed by a horizontal step, $\Psi(M^{\rev})$, a diagonal step, and finally a vertical step. In the special case $i=1$ when $M$ is empty, we obmit the diagonal step.

    If row~$i$ contains another nonzero entry (\cref{subfig:213-ASR_bij_b}), then it decomposes into $R^{\prime} \in \asr{r-i-1,r-i+1,r-i-1}{213}$ and a matrix~$M$ of size $i \times i$. The matrix~$M$ becomes an extendably $213$-avoiding ASM~$\widehat{M}$ if the $-1$ in its bottom left corner is replaced by $0$. We then map $R$ to $\Phi(R^{\prime})$, which is then followed by a horizontal step, $\Psi(\widehat{M}^{\rev})$, and a vertical step.

    To see that $\Phi$ is indeed bijective, we note two observations. First, the position (if any) of the $1$ in the first column corresponds to the last return of the little Schr\"oder path to the main diagonal. Secondly, since the matrix~$\widehat{M}$ in the case of \cref{subfig:213-ASR_bij_b} does not have a $1$ in its bottom left corner, $\Psi(\widehat{M}^{\rev})$ does not end with a diagonal step.
\end{proof}

\section{Concluding remarks}
\label{sec:conclusion}

We have proved recursive formulas for the number of ASRs extendably avoiding $\pi$ for $\pi=312,213,132,231$ and have conjectured an explicit formula in \cref{conj:312-formula}. A very natural question is if this can be extended to other patterns, which we briefly describe here.

\subsection{\texorpdfstring{Extendably avoiding the pattern~$\mathbf{123}$ or $\mathbf{321}$}{Extendably avoiding the pattern~123 or 321}}

In \cite{JL07}, it was shown that the numbers of $132$-avoiding ASMs---and, by symmetry, of $312$-, $213$-, and $231$-avoiding ASMs---are given by the large Schr\"oder numbers. However, for the remaining patterns of length~$3$, namely $123$ and $321$, no closed formula is known. The first few numbers of $123$-avoiding ASMs (and, by symmetry, of $321$-avoiding ASMs) are
$1,2,6,23,103,514,2785,16132,98897,637192,\dots$.

Similarly, enumerating extendably $123$- or $321$-avoiding ASRs remains an open problem. Obtaining a complete system of recurrence relations would be a step towards a closed form for the corresponding pattern-avoiding ASMs. One possibility is to add more parameters to find recursions for some refinements, 
but this has not been fruitful for us yet. 

\subsection{Extendably avoiding the cross pattern}

There are seven ASMs of size $3 \times 3$: the six permutation matrices and
\[
\chi = \begin{pmatrix} 0 & 1 & 0 \\ 1 & -1 & 1 \\ 0 & 1 & 0 \end{pmatrix},
\]
which we call the \emph{cross pattern}. There is a natural generalization of pattern avoidance with respect to the cross pattern: We say that an ASR is avoiding the cross pattern if there are no $1$s and $-1$s in the same relative position as in $\chi$.

Every ASR that contains a $-1$ also contains $\chi$ in every extension to an ASM. Thus, the number~$\s{\chi}{r}{k}{d}$ of $\chi$-avoiding ASRs of size $r \times k$ with $d$ nonempty rows equals the number of $\{0,1\}$-matrices of size $r \times k$ with exactly $d$ $1$s and at most one $1$ in each row and column. These are enumerated by
\begin{equation*}
    d! \binom{r}{d} \binom{k}{d}.
\end{equation*}

\bibliographystyle{alpha}
\bibliography{sample}

@unpublished{hkl,
AUTHOR = {H\"ongesberg, Hans and Konvalinka, Matja{\v z} and Linusson, Svante},
TITLE = {Pattern avoidance in alternating sign rectangles {II}: {C}lassical avoidance},
NOTE = {In preparation}
}

@book{AAR99,
    AUTHOR = {Andrews, George E. and Askey, Richard and Roy, Ranjan},
     TITLE = {Special functions},
    SERIES = {Encyclopedia of Mathematics and its Applications},
    VOLUME = {71},
 PUBLISHER = {Cambridge University Press, Cambridge},
      YEAR = {1999},
     PAGES = {xvi+664},
       DOI = {10.1017/CBO9781107325937},
}

@article {Bar21,
    AUTHOR = {Barry, Paul},
     TITLE = {Generalized {C}atalan recurrences, {R}iordan arrays, elliptic
              curves, and orthogonal polynomials},
   JOURNAL = {J. Integer Seq.},
  FJOURNAL = {Journal of Integer Sequences},
    VOLUME = {24},
      YEAR = {2021},
    NUMBER = {5},
     PAGES = {Art. 21.5.1, 48}
}

@article {Beh08,
    AUTHOR = {Behrend, Roger E.},
    TITLE = {Osculating paths and oscillating tableaux},
    JOURNAL = {Electron. J. Comb.},
    FJOURNAL = {Electronic Journal of Combinatorics},
    VOLUME = {15},
    YEAR = {2008},
    NUMBER = {1},
    PAGES = {Paper No. 7},
    DOI = {10.37236/731}
}

@article {BMO25,
    FJOURNAL = {Glasgow Mathematical Journal},
    JOURNAL = {Glasg. Math. J.},
    PAGES = {50--66},
    VOLUME = {67},
    NUMBER = {1},
    YEAR = {2025},
    TITLE = "{The combinatorics of $N_\infty$ operads for $C_{qp^n}$ and $D_{p^n}$}",
    AUTHOR = {Balchin, Scott and MacBrough, Ethan and Ormsby, Kyle},
    DOI = {10.1017/S0017089524000211}
}

@article{EL95,
    title={The size of {F}ulton’s essential set},
    volume={2},
    DOI={10.37236/1200},
    number={1},
    fjournal={The Electronic Journal of Combinatorics}, 
    JOURNAL = {Electron. J. Comb.},
    author={Eriksson, Kimmo and Linusson, Svante},
    year={1995},
    pages={R6}
}

@article{Fis06,
	author = {Fischer, Ilse},
	fjournal = {Advances in Applied Mathematics},
	journal = {Adv. Appl. Math.},
	number = {2},
	pages = {249--267},
	title = {The number of monotone triangles with prescribed bottom row},
	volume = {37},
	year = {2006},
	doi = {10.1016/j.aam.2005.03.009}
}

@article {For08,
    AUTHOR = {Fortin, Marc},
    TITLE = {The {M}ac{N}eille completion of the poset of partial injective functions},
    JOURNAL = {Electron. J. Comb.},
    FJOURNAL = {Electronic Journal of Combinatorics},
    VOLUME = {15},
    YEAR = {2008},
    NUMBER = {1},
    PAGES = {Paper No. R62},
    DOI = {10.37236/786}
}

@article {HS22,
    AUTHOR = {Heuer, Dylan and Striker, Jessica},
    TITLE = {Partial permutation and alternating sign matrix polytopes},
   JOURNAL = {SIAM J. Discrete Math.},
    FJOURNAL = {SIAM Journal on Discrete Mathematics},
    VOLUME = {36},
    YEAR = {2022},
    NUMBER = {4},
    PAGES = {2863--2888},
    DOI = {10.1137/21M1417958}
}

@article {JL07,
    AUTHOR = {Johansson, Robert and Linusson, Svante},
     TITLE = {Pattern avoidance in alternating sign matrices},
   JOURNAL = {Ann. Comb.},
  FJOURNAL = {Annals of Combinatorics},
    VOLUME = {11},
      YEAR = {2007},
    NUMBER = {3-4},
     PAGES = {471--480},
      ISSN = {0218-0006,0219-3094},
   MRCLASS = {05A15 (05B20)},
  MRNUMBER = {2376117},
MRREVIEWER = {Eric\ S.\ Egge},
       DOI = {10.1007/s00026-007-0330-1},
       URL = {https://doi.org/10.1007/s00026-007-0330-1},
}

@article {Kre73,
    AUTHOR = {Kreweras, Germain},
     TITLE = {Sur les hiérarchies de segments},
   JOURNAL = {Cahiers du Bureau universitaire de recherche opérationnelle. Série Recherche},
    VOLUME = {20},
     PAGES = {3--61},
      YEAR = {1973}
}

@article{Lin02,
    title = {Extended pattern avoidance},
   author = {Svante Linusson},
  journal = {Discrete Mathematics},
   volume = {246},
   number = {1--3},
    pages = {219--230},
     year = {2002},
      doi = {10.1016/S0012-365X(01)00259-X}
}

@misc{oeis,
    Author = {{OEIS Foundation Inc.}},
    Note = {Published electronically at \url{http://oeis.org}},
    Title = {The {O}n-{L}ine {E}ncyclopedia of {I}nteger {S}equences},
    Year = 2026}

@incollection {PR19,
    AUTHOR = {Pan, Ran and Remmel, Jeffrey B.},
     TITLE = {Paired patterns in lattice paths},
 BOOKTITLE = {Lattice path combinatorics and applications},
    SERIES = {Dev. Math.},
    VOLUME = {58},
     PAGES = {382--418},
 PUBLISHER = {Springer, Cham},
      YEAR = {2019},
       DOI = {10.1007/978-3-030-11102-1}
}

@article {PS98,
    AUTHOR = {Pergola, Elisa and Sulanke, Robert A.},
     TITLE = {Schr\"oder triangles, paths, and parallelogram polyominoes},
   JOURNAL = {J. Integer Seq.},
  FJOURNAL = {Journal of Integer Sequences},
    VOLUME = {1},
      YEAR = {1998},
     PAGES = {Article 98.1.7}
}

@article {Sch25,
    FJOURNAL = {Journal of Integer Sequences},
    JOURNAL = {J. Integer Seq.}, 
    PAGES = {Article 07.7.7},
    VOLUME = {10},
    YEAR = {2007},
    TITLE = {Generalized {S}chr\"oder numbers and the rotation principle},
    AUTHOR = {Schr\"oder, Joachim},
    URL = {https://cs.uwaterloo.ca/journals/JIS/VOL10/Schroder/schroder45.pdf}
}

@book{Knu73,
  author    = {Knuth, Donald E.},
  title     = {The Art of Computer Programming. {V}ol. 1:
               {F}undamental Algorithms},
  edition   = {Second},
  publisher = {Addison-Wesley Publishing Co., Reading, Mass.},
  year      = {1973},
}

@article{Kup96,
	author = {Kuperberg, Greg},
	fjournal = {International Mathematics Research Notices},
	journal = {	Int. Math. Res. Not.},
	number = {3},
	pages = {139--150},
	title = {Another proof of the alternating-sign matrix conjecture},
	volume = {1996},
	year = {1996},
	DOI = {10.1155/S1073792896000128}
}

@article{Zei96,
  author  = {Zeilberger, Doron},
  title   = {Proof of the alternating sign matrix conjecture},
  journal = {Electron. J. Comb.},
  volume  = {3},
  year    = {1996},
  number  = {2},
  pages   = {\#R13},
}

@article {BK26,
    AUTHOR = {Billey, Sara C. and Konvalinka, Matja\v{z}},
     TITLE = {Generalized rank functions and quilts of alternating sign
              matrices},
   JOURNAL = {Adv. Math.},
  FJOURNAL = {Advances in Mathematics},
    VOLUME = {491},
      YEAR = {2026},
     PAGES = {Paper No. 110863},
       DOI = {10.1016/j.aim.2026.110863},
       URL = {https://doi.org/10.1016/j.aim.2026.110863},
}

@article{BESST24,
	fjournal = {Annals of Combinatorics},
	journal = {Ann. Comb.},
	pages = {195--225},
	volume = {30},
	year = {2026},
	title = {Enumeration of Pattern-Avoiding Alternating Sign Matrices: An Asymptotic Dichotomy},
	author = {Bouvel, Mathilde and Egge, Eric S. and Smith, Rebecca N. and Striker, Jessica and Troyka, Justin M.},
	DOI = {10.1007/s00026-025-00784-9}
}

@book{Bre99,
  author    = {Bressoud, David M.},
  title     = {Proofs and Confirmations: The Story of the Alternating
               Sign Matrix Conjecture},
  series    = {MAA Spectrum},
  publisher = {Cambridge University Press, Cambridge; Mathematical
               Association of America, Washington, DC},
  year      = {1999},
  pages     = {xviii+274},
}

\end{document}